\documentclass{amsart}

\usepackage{paralist}
\usepackage[]{hyperref}
\usepackage{color}
\usepackage{pgf,tikz, tikzscale}
\usetikzlibrary{decorations.pathreplacing,decorations.markings}
\usetikzlibrary{arrows, knots, external, positioning}

\usetikzlibrary{arrows.meta}
\tikzstyle arrowstyle=[scale=2]
\tikzstyle mdirected=[postaction={decorate,decoration={markings,
    mark=at position .5 with {\arrow[arrowstyle]{stealth}}}}]
\tikzstyle edirected=[postaction={decorate,decoration={markings,
    mark=at position 1.0 with {\arrow[arrowstyle]{stealth}}}}]

\usetikzlibrary{circuits.logic.US,circuits.logic.IEC,fit}

\usepackage{tabularx}
\usepackage{array}
\newcolumntype{C}[1]{>{\centering\arraybackslash}m{#1}}

\newtheorem{thm}{Theorem}[section]

\newtheorem{cor}[thm]{Corollary}
\newtheorem{prop}[thm]{Proposition}

\theoremstyle{definition}
\newtheorem{defn}[thm]{Definition}

\numberwithin{equation}{section}

\newtheorem{theorem}{Theorem}[section]

\theoremstyle{definition}

\theoremstyle{remark}
\newtheorem{remark}[theorem]{Remark}
\newtheorem{example}[theorem]{Example}
\numberwithin{equation}{section}

\newcommand{\RN}[1]{%
  \textup{\uppercase\expandafter{\romannumeral#1}}%
}

\begin{document}

\title[Symplectic fillings of quotient singularities as Lefschetz fibrations]{A Lefschetz fibration on minimal symplectic fillings of a quotient surface singularity}

\author{Hakho Choi}
\address{Department of Mathematical Sciences, Seoul National University, Seoul 08826, Korea }
\email{hako85@snu.ac.kr}

\author{Jongil Park}
\address{Department of Mathematical Sciences, Seoul National University, Seoul 08826, Korea \& Korea Institute for Advanced Study, Seoul 02455, Korea}
\email{jipark@snu.ac.kr}

\thanks{}
\date{February 8, 2018; revised October 31, 2018}

\begin{abstract} 
In this article, we construct a genus-$0$ or genus-$1$ positive allowable Lefschetz fibration on any minimal symplectic filling of the link of non-cyclic quotient surface singularities. As a byproduct, we also show that any minimal symplectic filling of the link of quotient surface singularities can be obtained from a sequence of rational blowdowns from its minimal resolution.
\end{abstract}

\maketitle
\hypersetup{linkcolor=black}

\section{introduction}

Ever since S. Donaldson~\cite{D99} showed that any closed symplectic 4-manifold admits a Lefschetz pencil and that a Lefschetz fibration can be obtained from a Lefschetz pencil by blowing-up the base loci, the study of Lefschetz fibrations has
become an important theme for topologically understanding symplectic 4-manifolds.  
In fact, Lefschetz pencils and Lefschetz fibrations have been studied extensively by algebraic geometers and topologists in the complex category, and these notions can be extended to the symplectic category.
It is also known that an isomorphism class of Lefschetz fibrations is characterized by the monodromy factorization, an ordered sequence of right-handed Dehn twists, up to Hurwitz equivalence and global conjugation equivalence.  

On the other hand, a main research topic in symplectic 4-manifold topology focuses on classifying symplectic fillings of certain 3-manifolds equipped with a contact structure. Among them, people have classified symplectic fillings of the link of a quotient surface singularity. Note that the link of a quotient surface singularity admits a natural contact structure, called the \emph{Milnor fillable contact structure}.  
For example, P. Lisca~\cite{L} classified symplectic fillings of cyclic quotient singularities whose corresponding link is a lens space, and M. Bhupal and K. Ono~\cite{BOn} listed all possible symplectic fillings of non-cyclic quotient surface singularities. Furthermore, the second author together with H. Park, D. Shin, and G. Urz\'{u}a~\cite{PPSU} constructed an explicit one-to-one correspondence between the minimal symplectic fillings and the Milnor fibers of non-cyclic quotient surface singularities. Note that the last result above implies that every minimal symplectic filling of a quotient surface singularity is in fact a Stein filling. 

Although the existence of a (positive allowable) Lefschetz fibration, called briefly PALF, on a Stein filling is well known in general~\cite{AOz, LoPi}, it is a somewhat different problem to find an explicit monodromy description for the Lefschetz fibration on a given Stein filling.  
In this article, we investigate the problem for minimal symplectic fillings of the link of quotient surface singularities. 
M. Bhupal and B. Ozbagci~\cite{BOz} found an algorithm to present each minimal symplectic filling of a cyclic quotient surface singularity as an explicit genus-$0$ positive allowable Lefschetz fibration over the disk. Furthermore, they also showed that such a PALF can be obtained topologically from the minimal resolution by monodromy substitutions corresponding to rational blowdowns.
The main goal of this article is to generalize their result for the non-cyclic quotient surface singularity cases. Thus, we obtain the following result.

\begin{thm}
Every minimal symplectic filling of the link of non-cyclic quotient surface singularities admits a genus-$0$ or genus-$1$ positive allowable Lefschetz fibration over the disk. Furthermore, each symplectic filling can be also obtained by rational blowdowns from the minimal resolution of its singularity.
\label{mainthm}
\end{thm}
 
\begin{remark}
Note that a genus of the  PALF in Theorem~\ref{mainthm} above is determined only by the existence of a \emph{bad vertex} (refer to Section 2.1 for a definition) in the minimal resolution graph of the corresponding singularity. Explicitly, a genus of the PALF is $0$ if the minimal resolution graph of a quotient surface singularity has no bad vertex, and a genus is $1$ otherwise.  
\end{remark}
 
 In order to prove Theorem~\ref{mainthm} above, we first construct a PALF on the minimal resolution graph of a non-cyclic quotient surface singularity: If there is no bad vertex in the minimal resolution graph, we follow the idea of D. Gay and T. Mark in~\cite{GaM}, where they initially constructed a genus-$0$ PALF on the minimal resolution graph. If there is a bad vertex, then we construct a genus-$1$ PALF, which is a special case of open book decompositions on the boundary of plumbings obtained by J. Etnyre and B. Ozbagci~\cite{EtOz1}. Next, we show that the induced contact structure on the boundary is the Milnor fillable contact structure, which can be obtained by computing the first Chern class in terms of vanishing cycles and the rotation number of these vanishing cycles.
Then, we construct a PALF on any minimal symplectic filling via the corresponding $P$-resolution. Since every Milnor fiber, hence every minimal symplectic filling, of a quotient surface singularity can be obtained topologically by rationally blowing down the corresponding $P$-resolution, it is sufficient to construct a PALF on the general fiber of $P$-resolutions. 
Finally we show that a Lefschetz fibration of any minimal symplectic filling can be obtained by monodromy substitutions from the minimal resolution of the corresponding singularity by adapting the same technique  that H.~Endo, T.~Mark and J.~Van Horn-Morris used in~\cite{EnMV}.
\\
  
This article is organized as follows: We briefly review some generalities on quotient surface singularities, including minimal resolutions and $P$-resolutions, and the relation between monodormy substitutions and rational blowdowns in Section 2. We introduce Lisca's classification result on symplectic fillings and Bhupal-Ozbagci's algorithm for finding a PALF on the cyclic cases in Section 3. We subsequently explain how to construct a genus-$0$ or genus-$1$ Lefschetz fibration on the minimal resolutions and we show that the induced contact structure on the boundary is indeed Milnor fillable in Section 4. Finally, we provide an explicit algorithm for a PALF on any minimal symplectic filling by investigating a PALF on each $P$-resolution in Section 5.


\subsection*{Acknowledgements}

Jongil Park is supported by Samsung Science and Technology Foundation under Project Number SSTF-BA1602-02. 
He also holds a joint appointment at KIAS and in the Research Institute of Mathematics, SNU.




\section{Generalities on quotient surface singularities}
\label{Section-2}

In this section we briefly recall some basics on quotient surface singularities (refer to~\cite{PPSU} for details). 
Let $(X,0)=(\mathbb{C}^2/G, 0)$ be a germ of a quotient surface singularity, where $G$ is a finite subgroup of $GL(2, \mathbb{C})$ without reflections. Since $(\mathbb{C}^2/G_1, 0)$ is analytically isomorphic to $(\mathbb{C}^2/G_2, 0)$ if and only if $G_1$ is conjugate to $G_2$, it is enough to classify finite subgroups of $GL(2,\mathbb{C})$ without reflections up to conjugation when classifying quotient surface singularities $(\mathbb{C}^2/G,0)$. We may assume that $G \subset U(2)$ because $G$ is finite. The action of $G$ on $\mathbb{C}^2$ then lifts to an action on the blow-up of $\mathbb{C}^2$ at the origin. Thus, $G$ acts on the exceptional divisor $E \cong \mathbb{CP}^1$, where the action is induced by 
$G \subset U(2) \to PU(2) \cong SO(3)$. The image of $G$ in $SO(3)$ is either a (finite) cyclic subgroup, a dihedral group, tetrahedral group, octahedral group, or icosahedral group. Therefore quotient surface singularities are divided into five classes: \emph{cyclic} quotient surface singularities, dihedral singularities, tetrahedral singularities, octahedral singularities, and icosahedral singularities. We call the last four cases \emph{non-cyclic} quotient surface singularities.

\subsection{Symplectic fillings and Milnor fibers}

Let $(X,0)=(\mathbb{C}^2/G,0)$ be a germ of a quotient surface singularity, where $G$ is a finite subgroup of $U(2)$ without reflections. Assume that $(X, 0) \subset (\mathbb{C}^N, 0)$, which is always possible for a normal surface singularity. If $B \subset \mathbb{C}^N$ is a small ball centered at the origin, then a small neighborhood $X \cap B$ of the singularity is homeomorphic to the cone over its boundary $L := X \cap \partial B$. The smooth compact 3-manifold $L$ is called the \emph{link} of the singularity. It is well known that the topology of the germ $(X, 0)$ is completely determined by its link $L$ and the link $L$ admits a natural contact structure $\xi_{\textrm{st}}$, so-called \emph{Milnor fillable contact structure} $\xi_{\textrm{st}}=TL \cap JTL$, where $J$ is an induced complex structure along $L$. 
A \emph{(strong) symplectic filling} of $(X, 0)$ is a symplectic 4-manifold $(W, \omega)$, where the boundary $\partial W=L$ satisfies the compatibility condition $\omega=d\alpha_{\textrm{st}}$ near $L$, 
and where $\alpha_{\textrm{st}}$ is a 1-form defining the contact structure 
$\xi_{\textrm{st}}=\ker{\alpha_{\textrm{st}}}$ on $L$.
One may also define a so-called \emph{weak} symplectic filling. However, it is known that two notions of symplectic fillings coincide in our case because the link $L$ is a rational homology sphere. So we simply call them \emph{symplectic fillings}. 

Next, we call $W$ a \emph{Stein filling} of $(X,0)$ if it is a Stein manifold $W$ with $L$ as its strictly pseudoconvex boundary and $\xi_{\textrm{st}}$ is the set of complex tangencies to $L$. It is clear that Stein fillings are minimal symplectic fillings of the link $L$ of $(X,0)$.

Third, we call a proper flat map $\pi \colon \mathcal{X} \rightarrow \Delta$ with $\Delta=\{ t \in \mathbb{C} : |t| < \epsilon \}$ a \emph{smoothing} of $(X,0)$ if it satisfies $\pi^{-1}(0) = X$ and $\pi^{-1}(t)$ is smooth for all $t \neq 0$. The \emph{Milnor fiber} $M$ of a smoothing $\pi$ of $(X,0)$ is defined as a general fiber $\pi^{-1}(t)$ ($0 < t < \epsilon$). It is known that the Milnor fiber $M$ is a compact 4-manifold with link $L$ as its boundary and the diffeomorphism type depends only on the smoothing $\pi$. 
Furthermore, $M$ has a natural Stein (hence symplectic) structure, thus it provides an example of a Stein (and minimal symplectic) filling of $(L,\xi_{st})$.
Recall that, as mentioned in the Introduction, H. Park, J. Park, D. Shin, and G. Urz\'{u}a~\cite{PPSU} constructed an explicit one-to-one correspondence between the minimal symplectic fillings and the Milnor fibers of quotient surface singularities. Hence, it is now a well-known fact that every minimal symplectic filling of a quotient surface singularity is a Stein filling and a Milnor fiber of the singularity.

\subsection{Minimal resolutions}

We first denote the \emph{Hirzebruch-Jung continued fraction} by $[c_1,\dots,c_t] (c_i \geq 1)$, which is defined recursively as follows:
$$ [c_t] = c_t, \ \textrm{and}\ \ [c_i,c_{i+1},\dots,c_t]=c_i-\frac{1}{[c_{i+1},\dots,c_t]}.$$
Since a continued fraction $[c_1,c_2,\dots,c_t]$ often describes a chain of smooth rational curves on a complex surface whose dual graph is given by
\begin{center}
\begin{tikzpicture}
\filldraw (0,0) circle (2pt);
\filldraw (1,0) circle (2pt);
\filldraw (3,0) circle (2pt);
\filldraw (4,0) circle (2pt);
\draw (2,0) node {$\cdots$};
\draw (0,0) node[above] {$-c_1$};
\draw (1,0) node[above] {$-c_2$};
\draw (3,0) node[above] {$-c_{t-1}$};
\draw (4,0) node[above] {$-c_t$};
\draw (0,0)--(1,0);
\draw (3,0)--(4,0);
\draw (1,0)--(1.5,0) (2.5,0)--(3,0);
\end{tikzpicture}
\hspace{-.5 em},
\end{center}
we use by analogy the term `blowing up' for the following operations and the term `blowing down' for their inverses: 
\begin{eqnarray*}
\lbrack c_1, \dots, c_{i-1}, c_{i+1}, \dots, c_t]&\rightarrow&[c_1, \dots, c_{i-1}+1, 1, c_{i+1}+1, \dots, c_t\rbrack\\
\lbrack c_1, \dots, c_{t-1}  \rbrack&\rightarrow&\lbrack c_1, \dots, c_{t-1}+1, 1  \rbrack
.
\end{eqnarray*}

Now we describe the (dual graph of) the minimal resolution of quotient surface singularities. In the resolution graph, note that a vertex $v$ corresponds to the irreducible component $E_v$ of the exceptional divisor $E$, and the edges correspond to the intersections of the irreducible components $E_v$. We call the number of edges connected to the vertex v the \emph{valence} of v and the self-intersection of $E_v$ the \emph{degree} of v. If the absolute value of the degree of $v$ is strictly less than the valence of $v$, we call the vertex $v$ a \emph{bad vertex}.

\begin{example}
The following figures show the cases of minimal resolution graphs with and without a bad vertex. A central vertex (vertex with valence $3$) in the right-handed figure is a bad vertex.
\end{example}

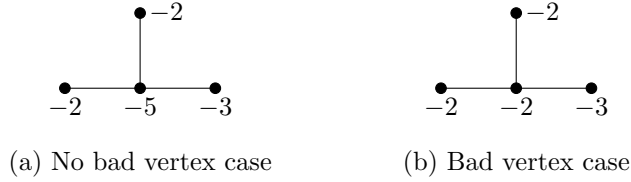
\begin{center}
\begin{tikzpicture}
\begin{scope}
\filldraw (0,0) circle (2pt);
\filldraw (1,0) circle (2pt);
\filldraw (2,0) circle (2pt);
\filldraw (1,1) circle (2pt);
\draw (0,0) node[below] {$-2$};
\draw (1,0) node[below] {$-5$};
\draw (2,0) node[below] {$-3$};
\draw (1,1) node[right] {$-2$};
\draw (0,0)--(2,0);
\draw (1,0)--(1,1);
\draw (1,-1) node {(a) No bad vertex case};
\end{scope}
\begin{scope}[shift={(5,0)}]
\filldraw (0,0) circle (2pt);
\filldraw (1,0) circle (2pt);
\filldraw (2,0) circle (2pt);
\filldraw (1,1) circle (2pt);

\draw (0,0) node[below] {$-2$};
\draw (1,0) node[below] {$-2$};
\draw (2,0) node[below] {$-3$};
\draw (1,1) node[right] {$-2$};
\draw (0,0)--(2,0);
\draw (1,0)--(1,1);
\draw (1,-1) node {(b) Bad vertex case};
\end{scope}
\end{tikzpicture}
\end{center}

\subsection*{Cyclic singularities $A_{n,q}$.}
$\phantom{0}$
A \emph{cyclic quotient surface singularity $(X,0)$ of type $\frac{1}{n}(1,q)$} with $1 \le q < n$ and $(n,q)=1$ is a quotient surface singularity, where a cyclic group $\mathbb{Z}_n$ acts by
%
$\zeta \cdot (x,y) = (\zeta x, \zeta^q y).$
%
Then, the minimal resolution graph of $(X,0)$ is given by
\begin{center}
\begin{tikzpicture}
\filldraw (0,0) circle (2pt);
\filldraw (1,0) circle (2pt);
\filldraw (3,0) circle (2pt);
\filldraw (4,0) circle (2pt);
\draw (2,0) node {$\cdots$};
\draw (0,0) node[above] {$-b_1$};
\draw (1,0) node[above] {$-b_2$};
\draw (3,0) node[above] {$-b_{r-1}$};
\draw (4,0) node[above] {$-b_r$};
\draw (0,0)--(1,0);
\draw (3,0)--(4,0);
\draw (1,0)--(1.5,0) (2.5,0)--(3,0);
\end{tikzpicture}
\hspace{-.5 em},
\end{center}
where $\displaystyle \frac{n}{q}=[b_1, b_2, \dots, b_{r-1}, b_r]$
with $b_i\geq 2$ for all $i$.

\subsection*{Dihedral singularities $D_{n,q}$.}
$\phantom{0}$
Let $(X,0)$ be a dihedral singularity of type $D_{n,q}$, where $1 < q < n$ and $(n,q)=1$. The minimal resolution graph of $(X,0)$ is given by
\begin{center}
\begin{tikzpicture}
\filldraw (-1,0) circle (2pt);
\filldraw (0,1) circle (2pt);

\filldraw (0,0) circle (2pt);
\filldraw (1,0) circle (2pt);
\filldraw (3,0) circle (2pt);
\filldraw (4,0) circle (2pt);
\draw (2,0) node {$\cdots$};
\draw (0,0) node[below] {$-b$};
\draw (1,0) node[below] {$-b_1$};
\draw (3,0) node[below] {$-b_{r-1}$};
\draw (4,0) node[below] {$-b_r$};
\draw (-1,0) node[below] {$-2$};
\draw (0,1) node[right] {$-2$};

\draw (0,0)--(-1,0) (0,0)--(0,1);
\draw (0,0)--(1,0);
\draw (3,0)--(4,0);
\draw (1,0)--(1.5,0) (2.5,0)--(3,0);
\end{tikzpicture}
\hspace{-.5 em},
\end{center}
where $\displaystyle \frac{n}{q}=[b, b_1, \dots, b_{r-1}, b_r]$ with $b \ge 2$ and $b_i \ge 2$ for all $i$.

\subsection*{Other cases.}
For a tetrahedral, octahedral, or icosahedral singularity, the minimal resolution has a central curve $C_0$ with $C_0 \cdot C_0 = -b$ ($b \ge 2$) and three arms, which can be divided into \emph{type $(3,1)$} and \emph{type $(3,2)$}:
\begin{center}
\begin{tikzpicture}
\begin{scope}
\filldraw (-1,0) circle (2pt);
\filldraw (0,1) circle (2pt);

\filldraw (0,0) circle (2pt);
\filldraw (1,0) circle (2pt);
\filldraw (3,0) circle (2pt);
\draw (2,0) node {$\cdots$};
\draw (0,0) node[below] {$-b$};
\draw (1,0) node[below] {$-b_1$};
\draw (3,0) node[below] {$-b_{r}$};
\draw (-1,0) node[below] {$-2$};
\draw (0,1) node[right] {$-3$};

\draw (0,0)--(-1,0) (0,0)--(0,1);
\draw (0,0)--(1,0);
\draw (1,0)--(1.5,0) (2.5,0)--(3,0);
\draw (1, -1) node {(a) type $(3,1)$};
\end{scope}
\begin{scope}[shift={(6.5,0)}]
\filldraw (-1,0) circle (2pt);
\filldraw (-2,0) circle (2pt);

\filldraw (0,1) circle (2pt);
\draw (0.5, -1) node {(b) type $(3,2)$};

\filldraw (0,0) circle (2pt);
\filldraw (1,0) circle (2pt);
\filldraw (3,0) circle (2pt);
\draw (2,0) node {$\cdots$};
\draw (0,0) node[below] {$-b$};
\draw (1,0) node[below] {$-b_1$};
\draw (3,0) node[below] {$-b_{r}$};
\draw (-1,0) node[below] {$-2$};
\draw (-2,0) node[below] {$-2$};

\draw (0,1) node[right] {$-2$};

\draw (0,0)--(-2,0) (0,0)--(0,1);
\draw (0,0)--(1,0);
\draw (1,0)--(1.5,0) (2.5,0)--(3,0);
\end{scope}
\end{tikzpicture}
\end{center}

\subsection{$P$-resolutions}

 
\begin{defn}
A normal surface singularity is \emph{of class $T$} if it is a rational double point singularity or a cyclic quotient surface singularity of type $\frac{1}{dn^2}(1, dna-1)$ with $d \ge 1$, $n \ge 2$, $1 \le a < n$, and $(n,a)=1$. Equivalently, it is a quotient surface singularity which admits a $\mathbb{Q}$-Gorenstein one-parameter smoothing.~\cite{KSB}
\end{defn}

Note that one-parameter $\mathbb{Q}$-Gorenstein smoothing of a singularity of class $T$ is interpreted topologically as a rational blowdown surgery defined by R.~Fintushel and R.~Stern~\cite{FS97}, and later extended by J.~Park~\cite{P}.
Furthermore, thanks to J. Wahl~\cite{W}, a cyclic quotient surface singularity of class $T$ can be recognized from its minimal resolution as follows:

\begin{prop}
\begin{enumerate}
\item The singularities  
\begin{tikzpicture} 
\filldraw (0,0) circle (2pt); 
\draw (0,0) node[above] {$-4$}; 
\end{tikzpicture} and 
\begin{tikzpicture} 
\filldraw (0,0) circle (2pt); 
\filldraw (1,0) circle (2pt); 

\filldraw (3,0) circle (2pt); 
\filldraw (4,0) circle (2pt); 

\draw (0,0)--(1.5,0) (2.5,0)--(4,0);
\draw (2,0) node {$\dots$};
\draw (0,0) node[above] {$-3$}; 
\draw (1,0) node[above] {$-2$}; 
\draw (3,0) node[above] {$-2$}; 
\draw (4,0) node[above] {$-3$}; 

\end{tikzpicture}
are of class $T$.

\item If 
\begin{tikzpicture} 
\filldraw (0,0) circle (2pt); 
\filldraw (1,0) circle (2pt); 

\filldraw (3,0) circle (2pt); 
\filldraw (4,0) circle (2pt); 

\draw (0,0)--(1.5,0) (2.5,0)--(4,0);
\draw (2,0) node {$\dots$};
\draw (0,0) node[above] {$-b_1$}; 
\draw (1,0) node[above] {$-b_2$}; 
\draw (3,0) node[above] {$-b_{r-1}$}; 
\draw (4,0) node[above] {$-b_r$};

\end{tikzpicture}
is of class $T$, so are $$\begin{tikzpicture}[scale=1.3] 
\filldraw (0,0) circle (2pt); 
\filldraw (1,0) circle (2pt); 

\filldraw (3,0) circle (2pt); 
\filldraw (4,0) circle (2pt); 

\draw (0,0)--(1.5,0) (2.5,0)--(4,0);
\draw (2,0) node {$\dots$};
\draw (0,0) node[above] {$-2$}; 
\draw (1,0) node[above] {$-b_1$}; 
\draw (3,0) node[above] {$-b_{r-1}$}; 
\draw (4,0) node[above] {$-(b_r+1)$};

\end{tikzpicture}$$
and
$$\begin{tikzpicture} [scale=1.3]
\filldraw (0,0) circle (2pt); 
\filldraw (1,0) circle (2pt); 

\filldraw (3,0) circle (2pt); 
\filldraw (4,0) circle (2pt); 

\draw (0,0)--(1.5,0) (2.5,0)--(4,0);
\draw (2,0) node {$\dots$};
\draw (0,0) node[above] {$-(b_1+1)$}; 
\draw (1,0) node[above] {$-b_2$}; 
\draw (3,0) node[above] {$-b_{r}$}; 
\draw (4,0) node[above] {$-2$};

\end{tikzpicture}$$
\item Every singularity of class $T$ that is not a rational double point can be obtained directly from one of the singularities described in (1) and by iterating through the steps described in (2) above.
\end{enumerate}
\end{prop}

\begin{defn}
A \emph{$P$-resolution} $f: (Y,E)\rightarrow (X,0)$ of a quotient surface singularity $(X,0)$ is a partial resolution such that $Y$ has at most rational double points or singularities of class $T$ and $K_Y$ is ample relative to $f$.
\end{defn}

 We usually describe a $P$-resolution $Y\rightarrow X$ as the minimal resolution $\pi:Z\rightarrow Y$ of $Y$ with $\pi$-exceptional divisors. Note that the ampleness condition in the definition of a $P$-resolution can be checked on $Z$: Every $(-1)$ curve on $Z$ must intersect two curves $E_1$ and $E_2$, which are exceptional for singularities of class $T$ on $Y$. In addition, the sum of the $k_i$ coefficients of $E_i$ in the canonical divisor $K_{Z}$ is less than $-1$. 
 According to J. Kollar and N. Shepherd-Barron~\cite{KSB}, there is a one-to-one correspondence between the set of all irreducible components of the versal deformation space of a quotient surface singularity $(X,0)$ and the set of all $P$-resolutions of $(X,0)$. 
Hence, since the Milnor fibers are invariants of the irreducible components of the versal deformation space of $(X,0)$, there is a one-to-one correspondence between the Milnor fibers and the $P$-resolutions. Furthermore, J. Stevens~\cite{Ste2} also showed how to find all $P$-resolutions of quotient surface singularities.

 \begin{example}
 Let $(X,0)$ be a dihedral singularity of type $D_{9,2}$. Since $9/2=[5,2]$, the minimal resolution of $(X,0)$ is given by
 $$
  
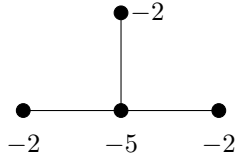
\begin{tikzpicture}[scale=1.3]
 \filldraw (0,0) circle (2pt);
 \filldraw (1,0) circle (2pt);
 \filldraw (2,0) circle (2pt);
 \filldraw (1,1) circle (2pt); 
\draw (0,0)--(2,0) (1,0)--(1,1);

\draw (0,-0.15) node[below] {$-2$};
\draw (1,1) node[right] {$-2$};
\draw (1,-0.15) node[below] {$-5$};
\draw (2,-0.15) node[below] {$-2$};
\end{tikzpicture}
 $$
We have the following four $P$-resolutions of $(X,0)$:
Here, a linear chain of vertices decorated by a rectangle $\square$ denotes curves on the minimal resolution of a $P$-resolution, which are contracted to a singularity of class $T$ on the $P$-resolution.
Note that there are certain symmetries in the list of $P$-resolutions. 
$$
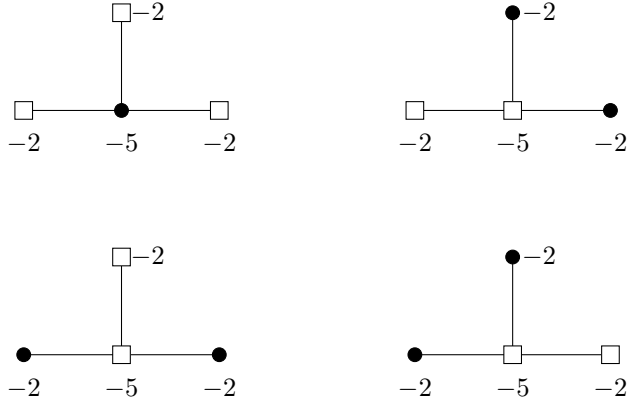
\begin{tikzpicture}[scale=1.3]
\begin{scope}
\draw (0,0)--(2,0) (1,0)--(1,1);
\node [draw, fill=white, shape=rectangle, anchor=center] at (0,0) {};
\node [draw, fill=white, shape=rectangle, anchor=center] at (1,1) {};
\node [draw, fill=white, shape=rectangle, anchor=center] at (2,0) {};
 \filldraw (1,0) circle (2pt);

\draw (1,1) node[right] {$-2$};
\draw (0,-0.15) node[below] {$-2$};
\draw (1,-0.15) node[below] {$-5$};
\draw (2,-0.15) node[below] {$-2$};
\end{scope}

\begin{scope}[shift={(4,0)}]
\draw (0,0)--(2,0) (1,0)--(1,1);
\node [draw, fill=white, shape=rectangle, anchor=center] at (0,0) {};
\node [draw, fill=white, shape=rectangle, anchor=center] at (1,0) {};
 \filldraw (1,1) circle (2pt);
 \filldraw (2,0) circle (2pt);

\draw (1,1) node[right] {$-2$};
\draw (0,-0.15) node[below] {$-2$};
\draw (1,-0.15) node[below] {$-5$};
\draw (2,-0.15) node[below] {$-2$};
\end{scope}

\begin{scope}[shift={(0,-2.5)}]
\draw (0,0)--(2,0) (1,0)--(1,1);
\node [draw, fill=white, shape=rectangle, anchor=center] at (1,0) {};
\node [draw, fill=white, shape=rectangle, anchor=center] at (1,1) {};
 \filldraw (0,0) circle (2pt);
 \filldraw (2,0) circle (2pt);

\draw (1,1) node[right] {$-2$};
\draw (0,-0.15) node[below] {$-2$};
\draw (1,-0.15) node[below] {$-5$};
\draw (2,-0.15) node[below] {$-2$};
\end{scope}

\begin{scope}[shift={(4,-2.5)}]
\draw (0,0)--(2,0) (1,0)--(1,1);
\node [draw, fill=white, shape=rectangle, anchor=center] at (1,0) {};
\node [draw, fill=white, shape=rectangle, anchor=center] at (2,0) {};
 \filldraw (0,0) circle (2pt);
 \filldraw (1,1) circle (2pt);

\draw (1,1) node[right] {$-2$};
\draw (0,-0.15) node[below] {$-2$};
\draw (1,-0.15) node[below] {$-5$};
\draw (2,-0.15) node[below] {$-2$};
\end{scope}

\end{tikzpicture}
$$
 \end{example}

\subsection{Monodromy substitutions and rational blowdowns}

In ~\cite{FS97}, R. Fintushel and R. Stern introduced the following rational blowdown surgery: Let $C_p$ be a smooth $4$-manifold obtained from plumbing disk bundles over $2$-sphere according to the following linear diagram.
\begin{center}
\begin{tikzpicture}
\filldraw (0,0) circle (2pt);
\filldraw (1,0) circle (2pt);
\filldraw (3,0) circle (2pt);
\filldraw (4,0) circle (2pt);
\draw (2,0) node {$\cdots$};
\draw (-0.2,0) node[above] {$-(p+2)$};
\draw (1,0) node[above] {$-2$};
\draw (3,0) node[above] {$-2$};
\draw (4,0) node[above] {$-2$};
\draw (0,0)--(1,0);
\draw (3,0)--(4,0);
\draw (1,0)--(1.5,0) (2.5,0)--(3,0);
\end{tikzpicture}
\end{center}
Then the boundary of $C_p$ is a lens space $L(p^2,p-1)$, which bounds a rational ball $B_p$, i.e., $H_*(B_p; \mathbb{Q})=H_*(D^4; \mathbb{Q})$. Therefore, if there exists an embedding $C_p$ in a smooth $4$-manifold $X$, one can construct a new smooth $4$-manifold $X_p$ by replacing $C_p$ with $B_p$. This procedure is called a \emph{rational blowdown} surgery and we say that $X_p$ is obtained by rationally blowing down $X$. 
Furthermore, M. Symington~\cite{Sy} proved that a rational blowdown $4$-manifold
 $X_{p}$ admits a symplectic structure in some cases.
 For example, if $X$ is a symplectic $4$-manifold containing
 a configuration $C_{p}$ such that all $2$-spheres in $C_{p}$
 are symplectically embedded and intersect positively,
 then the rational blowdown $4$-manifold $X_{p}$ also admits a symplectic structure.
Later, the Fintushel-Stern's rational blowdown surgery is generalized by J. Park~\cite{P} using a configuration $C_{p,q}$ obtained from plumbing disk bundles over a $2$-sphere according to the dual resolution graph of $L(p^2,pq-1)$, which also bounds a rational ball $B_{p,q}$ as follows:

\begin{defn}
 Suppose $X$ is a smooth $4$-manifold
 containing a configuration $C_{p,q}$. Then one can construct
 a new smooth $4$-manifold $X_{p,q}$,
 called a {\em $($generalized$)$ rational blowdown} of $X$,
 by replacing $C_{p,q}$ with a rational ball $B_{p,q}$.
 We also call this a {\em (generalized) rational blowdown} surgery.
\end{defn}

Next, we introduce a notion of monodromy substitution that is closely related to a rational blowdown surgery. That is, we briefly explain how to replace a rational blowdown surgery with a monodromy substitution in some cases. 

Suppose that a symplectic $4$-manifold $X$ with a possibly non-empty boundary admits a Lefschetz fibration characterized by a monodromy factorization $\mathcal{W}_X$. Assume that $W$ and $W'$ are distinct products of right-handed Dehn twists which yield the same element as a global monodromy in the mapping class group of the fiber. If there is a partial monodromy factorization equal to $W$ in the monodromy factorization $\mathcal{W}_X$ of $X$, then we can obtain a Lefschetz fibration on a new symplectic $4$-manifold $X'$ whose monodromy factorization $\mathcal{W}_{X'}$ is obtained by replacing $W$ with $W'$. 
Note that the diffeomorphism types and the induced contact structures of $\partial X$ and $\partial X'$ are the same. We call this procedure a \emph{monodromy substitution}. 
For example, a famous lantern relation yields a rational blowdown surgery involving the lens space $L(4,1)$~\cite{EnGu}: the PALF with monodromy $abcd$ yields a configuration $C_2$, while the PALF with monodromy $xyz$ yields a rational ball $B_2$. As another example, the \emph{daisy relation}, introduced in~\cite{EnMV}, yields a monodromy substitution for a configuration $C_{p}$ and a rational ball $B_{p}$. One can also find a monodromy substitution for a (generalized) rational blowdown surgery in~\cite{EnMV}.

\begin{figure}[h]
\begin{center}
\begin{tikzpicture}[scale=1.3]
\begin{scope}
\draw (3,-3) ellipse (2 and 1.2);
\draw (3,-3) circle (0.15);
\draw (4,-3) circle (0.15);
\draw (2,-3) circle (0.15);
\draw[red] (3,-3) circle (0.35);
\draw[red] (4,-3) circle (0.35);
\draw[red] (2,-3) circle (0.35);
\draw[red] (3,-3) ellipse (1.8 and 0.8);

\draw (2,-3.35) node[below] {$a$};
\draw (3,-3.35) node[below] {$b$};
\draw (4,-3.35) node[below] {$c$};
\draw (3.45,-4) node {$d$};
\end{scope}
\begin{scope}[shift={(5,0)}]
\draw (3,-3) ellipse (2 and 1.2);
\draw (3,-3) circle (0.15);
\draw (4.25,-3) circle (0.15);
\draw (1.75,-3) circle (0.15);
\draw[red] (1.5,-3) arc (180:360: 0.25);
\draw[red] (4,-3) arc (180:360: 0.25);
\draw[red] (4.5,-3) to [out=up, in=up] (1.5,-3);
\draw[red] (2,-3) to [out=up, in=up] (4,-3);
\draw[red] (3.675,-3) ellipse (1.1 and 0.4); 
\draw[red] (2.375,-3) ellipse (1.1 and 0.4); 
\draw (1.75,-3.5) node {$x$};
\draw (3,-2) node {$z$};
\draw (4.25,-3.5) node {$y$};

\end{scope}
\end{tikzpicture}
\end{center}
\caption{Lantern relation.}
\end{figure}
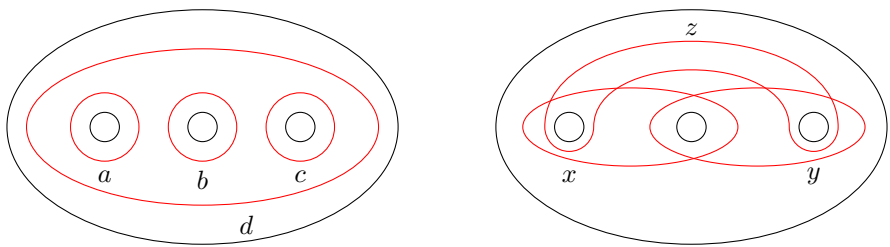


\section{Review for the cyclic singularity cases}

We briefly review P.~Lisca's classification of minimal symplectic fillings and Bhupal-Ozbagci's algorithm of positive allowable Lefschetz fibrations for the minimal symplectic fillings of a cyclic quotient surface singularity in this section.

We first review Lisca's classification (refer to~\cite{L} for details): Let $(X,0)$ be a cyclic quotient surface singularity of type $\frac{1}{n}(1,q)$ with $(n,q)=1$ whose link is the lens space $L(n,q)$. P. Lisca~\cite{L} parametrized all minimal symplectic fillings of $(X,0)$ by a set $\mathcal{Z}_e(\displaystyle\frac{n}{n-q})$ of certain sequences of integers $\mathbf{n}=(n_1, \dotsc, n_e) \in \mathbb{N}^e$ (see Definition~\ref{defn-Z_e} below). That is, he constructed a compact oriented symplectic 4-manifold $W_{n,q}(\mathbf{n})$ with boundary $L(n,q)$ that is parametrized by $\mathbf{n} \in \mathcal{Z}_e(\displaystyle\frac{n}{n-q})$ using surgery diagrams. He also showed that $W_{n,q}(\mathbf{n})$ is in fact a Stein filling of $L(n,q)$. Finally, he proved that any symplectic filling of $L(n,q)$ is orientation-preserving diffeomorphic to a manifold obtained by blow-ups from one of $W_{n,q}(\mathbf{n})$'s. 
Hence, every minimal symplectic filling is diffeomorphic to 
one of the $W_{n,q}(\mathbf{n})$'s.
On the other hand, J.~Christophersen~\cite{C} and J.~Stevens~\cite{Ste1} parametrized all reduced irreducible components of the versal deformation space of $(X, 0)$ using the same set $\mathcal{Z}_e(\displaystyle\frac{n}{n-q})$ but with different methods. Thus, it was a natural conjecture that every Milnor fiber of $(X,0)$ is diffeomorphic to a $W_{n,q}(\mathbf{n})$, which are parametrized by the same element in $\mathcal{Z}_e(\displaystyle\frac{n}{n-q})$.
The conjecture was proven true by A.~Nemethi and P.~Popescu-Pampu~\cite{NPo}.

\begin{defn}
An $e$-tuple of nonnegative integers $(n_1\dots, n_e)$ is called \emph{admissible} if every denominator in the continued fraction $[n_1\dots, n_e]$ is positive.
It is easy to see that an admissible $e$-tuple of nonnegative integers is either $0$ or only consists of positive integers. Let $\mathcal{Z}_e$ be the set of all admissible $e$-tuples such that $[n_1\dots, n_e]=0$, i.e., let $\mathcal{Z}_e$ be the set of all $e$-tuples of integers which can be obtained via a sequence of blow-ups from $(0)$. For $\displaystyle \frac{n}{n-q}=[a_1,\dots, a_e]$, we define 
 $$\mathcal{Z}_e(\displaystyle\frac{n}{n-q}):=\{(n_1,\dots, n_e)\in\mathcal{Z}_e\vert\phantom{0} 0\leq n_i \leq a_i, \
 \textrm{for} \ i=1, \dots, e \}.$$ 
\label{defn-Z_e}
\end{defn}

 P. Lisca constructed a smooth $4$-manifold $W_{n,q}(\mathbf{n})$ for each $e$-tuple $\mathbf{n} \in \mathcal{Z}_e(\displaystyle\frac{n}{n-q})$ whose boundary is diffeomorphic to the link of a cyclic singularity of type $A_{n,q}$, also known as the lens space $L(n,q)$ using a corbodism $C_{n,q}(\mathbf{n})$ between $S^1\times S^2$ and $L(n,q)$:
 First consider a linear chain consisting of $e$ number of unknots in $S^3$ with framings $n_1, \dots, n_e$, respectively. Let $N(\mathbf{n})$ be a $3$-manifold obtained by Dehn surgery on this framed link. Since $[n_1,\dots, n_e]=0$, it is clear that $N(\mathbf{n})$ is diffeomorphic to $S^1\times S^2$. 
Then, using a framed link $L$ in $N(\mathbf{n})$ as shown in Figure~\ref{link}, one can obtain a cobordism $C_{n,q}(\mathbf{n})$ by attaching $4$-dimensional $2$-handles to the $L \subset S^1 \times S^2 \times \{1\} \subset S^1\times S^2 \times I$. 
Finally, choosing a diffeomorphism $\varphi : N(\mathbf{n}) \rightarrow S^1\times S^2$ again,
one can construct a desired smooth (in fact symplectic) $4$-manifold
$$W_{n,q}(\mathbf{n}): =C_{n,q}(\mathbf{n})\cup_\varphi S^1\times D^3.$$
Note that, since any self-diffeomorphism $\varphi$ of $S^1\times S^2$ extends to $S^1\times D^3$, the diffeomorphism type of $W_{n,q}(\mathbf{n})$ is independent of the choice of $\varphi$.
According to P. Lisca~\cite{L}, any symplectic filling of $(L(n,q),\xi_{st})$ is orientation-preserving diffeomorphic to a blow-up of $W_{n,q}(\mathbf{n})$ for some $\mathbf{n} \in \mathcal{Z}_e(\displaystyle\frac{n}{n-q})$.

\begin{figure}[h]
\begin{center}
\begin{tikzpicture}[scale=0.4]
\begin{knot}[
	clip width=5,
	clip radius = 2pt,
	end tolerance = 1pt,
]
\strand (0,0) ellipse (3 and 1.5);
\strand (5,0) ellipse (3 and 1.5);
\strand  (10,1.5) arc (90:270: 3 and 1.5);

\strand (14,1.5) arc (90:-90: 3 and 1.5);
\strand (19,0) ellipse (3 and 1.5);
\strand (24,0) ellipse (3 and 1.5);

\strand[red] (-1.00,-1.75) ellipse (0.25 and 0.8);
\strand[red] (-0.25,-1.75) ellipse (0.25 and 0.8);
\draw (0.5,-2) node {$\cdots$};
\strand[red] (1.25,-1.75) ellipse (0.25 and 0.8);
\draw (-1,-2.55) node[below] {$-1$};
\draw (-0.05,-2.55) node[below] {$-1$};
\draw (1.25,-2.55) node[below] {$-1$};
	\draw [decorate,decoration={brace,mirror,amplitude=5pt},xshift=0pt,yshift=-10pt]
	(-1,-3.25) -- (1.25,-3.25) node [black,midway,yshift=-10pt] 
	{\footnotesize $a_1-n_1$};

\strand[red] (4.00,-1.75) ellipse (0.25 and 0.8);
\strand[red] (4.75,-1.75) ellipse (0.25 and 0.8);
\draw (5.5,-2) node {$\cdots$};
\strand[red] (6.25,-1.75) ellipse (0.25 and 0.8);
\draw (4,-2.55) node[below] {$-1$};
\draw (4.95,-2.55) node[below] {$-1$};
\draw (6.25,-2.55) node[below] {$-1$};
	\draw [decorate,decoration={brace,mirror,amplitude=5pt},xshift=0pt,yshift=-10pt]
	(4,-3.25) -- (6.25,-3.25) node [black,midway,yshift=-10pt] 
	{\footnotesize $a_2-n_2$};

\strand[red] (18.00,-1.75) ellipse (0.25 and 0.8);
\strand[red] (18.75,-1.75) ellipse (0.25 and 0.8);
\draw (19.5,-2) node {$\cdots$};
\strand[red] (20.25,-1.75) ellipse (0.25 and 0.8);
\draw (18,-2.55) node[below] {$-1$};
\draw (18.95,-2.55) node[below] {$-1$};
\draw (20.25,-2.55) node[below] {$-1$};
	\draw [decorate,decoration={brace,mirror,amplitude=5pt},xshift=0pt,yshift=-10pt]
	(18,-3.25) -- (20.25,-3.25) node [black,midway,yshift=-10pt] 
	{\footnotesize $a_{e-1}-n_{e-1}$};

\strand[red] (23.00,-1.75) ellipse (0.25 and 0.8);
\strand[red] (23.75,-1.75) ellipse (0.25 and 0.8);
\draw (24.5,-2) node {$\cdots$};
\strand[red] (25.25,-1.75) ellipse (0.25 and 0.8);
\draw (23,-2.55) node[below] {$-1$};
\draw (23.95,-2.55) node[below] {$-1$};
\draw (25.25,-2.55) node[below] {$-1$};
	\draw [decorate,decoration={brace,mirror,amplitude=5pt},xshift=0pt,yshift=-10pt]
	(23,-3.25) -- (25.25,-3.25) node [black,midway,yshift=-10pt] 
	{\footnotesize $a_e-n_e$};

\flipcrossings{1,9,17,19,3,5,7,11,13,15,21,23,25,27,29,31}
\end{knot}
\draw (12,0) node {$\cdots$};
\draw (0,1.5) node[above] {$n_1$};
\draw (5,1.5) node[above] {$n_2$};
\draw (19,1.5) node[above] {$n_{e-1}$};
\draw (24,1.5) node[above] {$n_e$};
\end{tikzpicture}

\end{center}
\caption{The framed link $L \subset N(\mathbf{n})$. }
\label{link}
\end{figure}
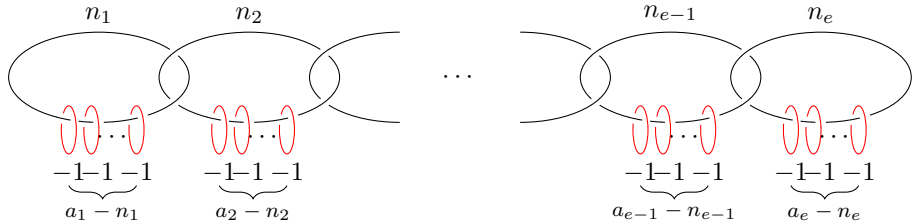

Next, for each $\mathbf{n} \in\mathcal{Z}_e(\displaystyle\frac{n}{n-q})$, M.~Bhupal and B.~Ozbagci constructed a genus-$0$ PALF on $S^1 \times D^3$ so that the attaching circles of $(-1)$-framed $2$-handles in $W_{n, q}(\mathbf{n})$ lie on a generic fiber (refer to~\cite{BOz} for details).

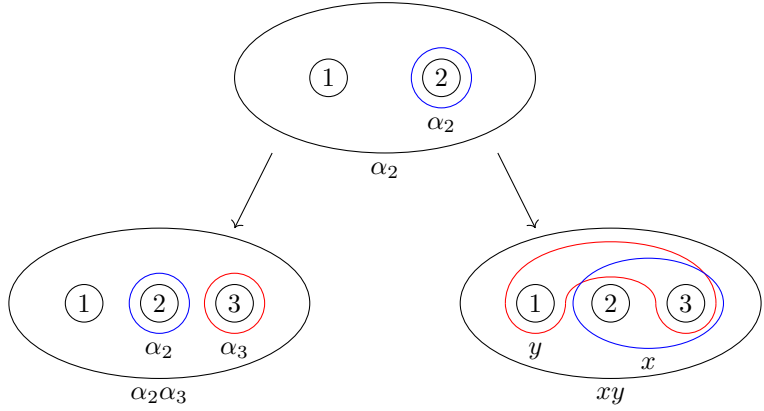
\begin{figure}[h]
\begin{center}
\begin{tikzpicture}[scale=1]
\draw (0,0) ellipse (2 and 1);
\draw (-0.75,0) circle (0.25);
\draw (0.75,0) circle (0.25);
\draw[blue] (0.75,0) circle (0.4);
\draw (0.75,-0.4) node[below] {$\alpha_2$};
\draw (-0.75,0) node {$1$};
\draw (0.75,0) node {$2$};
\draw (0,-1) node[below] {$\alpha_2$};
\draw (-3,-3) ellipse (2 and 1);
\draw (-3,-4) node[below] {$\alpha_2 \alpha_3$};
\draw (-3,-3) circle (0.25);
\draw (-4,-3) circle (0.25);
\draw (-2,-3) circle (0.25);
\draw (-4,-3) node {$1$};
\draw (-3,-3) node {$2$};
\draw (-2,-3) node {$3$};
\draw[blue] (-3,-3) circle (0.4);
\draw[red] (-2,-3) circle (0.4);
\draw (-3,-3.4) node[below] {$\alpha_2$};
\draw (-2,-3.4) node[below] {$\alpha_3$};
\draw (3,-3) ellipse (2 and 1);
\draw (3,-4) node[below] {$xy $};
\draw (3,-3) circle (0.25);
\draw (4,-3) circle (0.25);
\draw (2,-3) circle (0.25);
\draw (2,-3) node {$1$};
\draw (3,-3) node {$2$};
\draw (4,-3) node {$3$};
\draw[red] (1.6,-3) arc (180:360: 0.4);
\draw[red] (3.6,-3) arc (180:360: 0.4);
\draw[red] (3.6,-3) to [out=up, in=up] (2.4,-3);
\draw[red] (1.6,-3) to [out=up, in=up] (4.4,-3);
\draw[blue] (3.5,-3) ellipse (1 and 0.6); 
\draw (2,-3.4) node[below] {$y$};
\draw (3.5,-3.6) node[below] {$x$};
\draw[->]  (-1.5, -1)--(-2,-2);
\draw[->]  (1.5, -1)--(2,-2);
\end{tikzpicture}
\end{center}
\caption{PALF on $S^1\times D^3$ corresponds to $(1,1)$, $(1,2,1)$ and $(2,1,2)$.}
\label{s1s2}
\end{figure}

One can construct a PALF on $S^1 \times D^3$ over the disk corresponding to each $\mathbf{n} \in \mathcal{Z}_e(\displaystyle\frac{n}{n-q})$. Note that this depends on a blow-up sequence from $(0)$. 
For each $\mathbf{n} \in \mathcal{Z}_e(\displaystyle\frac{n}{n-q})$, a generic fiber $F_{\mathbf{n}}$ is the disk with $e$ holes. We may assume the holes in the disk are ordered linearly from left to right, as shown in Figure~\ref{s1s2}. 
If $\mathbf{n} \in \mathcal{Z}_e(\displaystyle\frac{n}{n-q})$ is obtained from $\mathbf{n'}\in \mathcal{Z}_{e-1}$ by blowing up the $j^{\textrm{th}}$ term $(1\leq j \leq e-2)$, we construct a generic fiber $F_{\mathbf{n}}$ to be a surface obtained from $F_{\mathbf{n'}}$ by splitting the $(j+1)^{\textrm{th}}$ hole so that vanishing cycles $\{x_i\, |\, i=1,2,\dots ,e-2\}$ for $\mathbf{n'}$ are naturally extended to $\{\widetilde{x_i}\, |\, i=1,2,\dots ,e-2\}$ in $F_{\mathbf{n}}$. 
The monodromy factorization subsequently changes from $x_1x_2\cdots x_{e-2}$ to $\widetilde{x_1}\widetilde{x_2}\cdots \widetilde{x}_{e-2}\beta_j$, where $\beta_j$ is a curve on $F_{\mathbf{n}}$ that encircles the $1$, $\dots j$, $(j+2)$-labelled holes while skipping the $(j+1)$-labelled hole. 
For a blowing up of the $(e-1)^{\textrm{th}}$ term, we just add the $e^{\textrm{th}}$ hole to $F_{\mathbf{n'}}$ at the right of the $(e-1)^{\textrm{th}}$ hole and add a Dehn twist on a curve encircling the $e^{\textrm{th}}$ hole. 
In this way, we obtain a genus-$0$ PALF on $W_{n,q}(\mathbf{n})$ such that, if the attaching circle of a $(-1)$-framed $2$-handle $h$ in $W_{n,q}(\mathbf{n})$ is the meridian of a $n_i$-framed unknot,
the $2$-handle $h$ corresponds to a Dehn twist on a curve $\gamma_i$ encircling the first $i$ holes. 
We refer to Figure~\ref{W92} below for an example.
\\

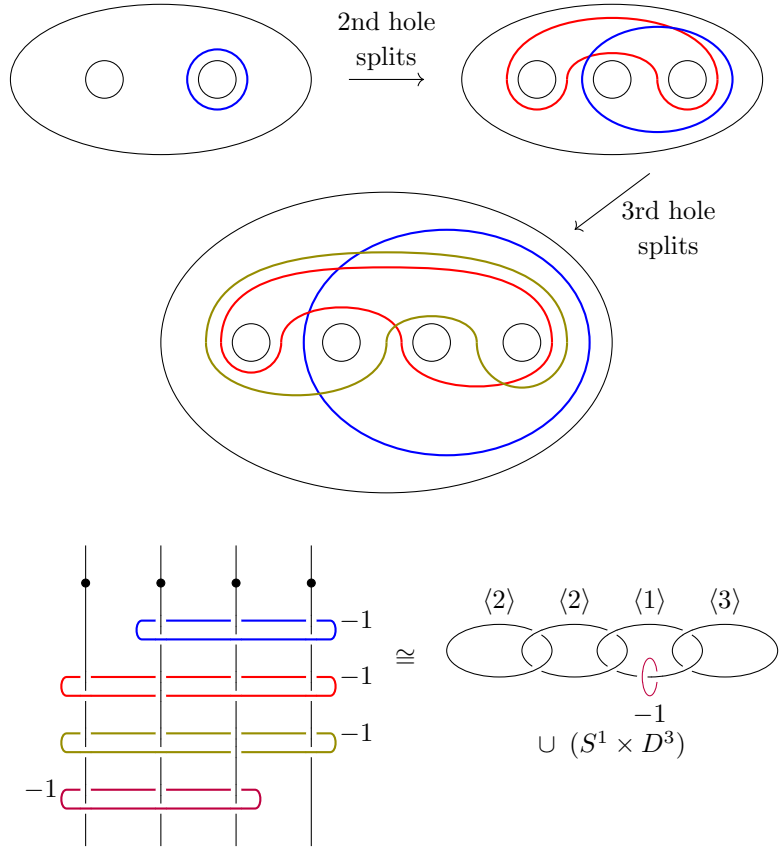
\begin{figure}[htb]
\begin{center}
\begin{tikzpicture}[scale=1]
\begin{scope}

\draw (0,0) ellipse (2 and 1);
\draw (-0.75,0) circle (0.25);
\draw (0.75,0) circle (0.25);
\draw[blue, thick] (0.75,0) circle (0.4);

\draw[->] (2.5,0)-- (3.5,0);
\draw[->] (6.5, -1.25)--(5.5,-2);
\draw (6.75,-2) node[text width=3cm,align=center] {3rd hole \\ splits};
\draw (3,0) node[above, text width=3cm,align=center] {2nd hole \\ splits};

\draw (6,0) ellipse (2 and 1);
\draw (6,-0) circle (0.25);
\draw (7,-0) circle (0.25);
\draw (5,-0) circle (0.25);
\draw[red, thick] (4.6,-0) arc (180:360: 0.4);
\draw[red, thick] (6.6,-0) arc (180:360: 0.4);
\draw[red, thick] (6.6,-0) to [out=up, in=up] (5.4,-0);
\draw[red, thick] (4.6,-0) to [out=up, in=up] (7.4,-0);
\draw[blue, thick] (6.6,-0) ellipse (1 and 0.7);

\draw (3,-3.5) ellipse (3 and 2);
\draw (1.2,-3.5) circle (0.25);
\draw (2.4,-3.5) circle (0.25);
\draw (3.6,-3.5) circle (0.25);
\draw (4.8,-3.5) circle (0.25);

\draw[blue, thick] (3.8,-3.5) ellipse (1.9 and 1.5);
\draw[red, thick] (1.6,-3.50) arc (0:-180: 0.4);
\draw[red, thick] (5.2,-3.50) to [out=down, in=down] (3.2,-3.5);
\draw[red, thick] (1.6,-3.50) to [out=up, in=up] (3.2,-3.5);
\draw[red, thick] (0.8,-3.50) to [out=up, in=left] (3,-2.5) to [out=right, in=up] (5.2,-3.5);

\draw[olive, thick] (0.6,-3.5) to [out=down, in=down] (3,-3.5);
\draw[olive, thick] (0.6,-3.50) to [out=up, in=left] (3,-2.3) to [out=right, in=up] (5.4,-3.5);
\draw[olive, thick] (5.4,-3.5) arc (0:-180:0.6);
\draw[olive, thick] (3,-3.50) to [out=up, in=up] (4.2,-3.5);

\end{scope}

\begin{scope}[shift={(-1,-6.7)}]
\begin{knot}[
	clip width=5,
	clip radius = 2pt,
	end tolerance = 1pt,
]
\strand (0,0.5)--(0,-3.5);
\strand (1,0.5)--(1,-3.5);
\strand (2,0.5)--(2,-3.5);
\strand (3,0.5)--(3,-3.5);
\strand[blue, thick] (3.25,-0.5) to [out=right, in=right] (3.25,-0.75)--(0.75,-0.75) to [out=left, in=left] (0.75,-0.5)--cycle;
\strand[red, thick] (3.25,-1.25) to [out=right, in=right] (3.25,-1.5)--(-0.25,-1.5) to [out=left, in=left] (-0.25,-1.25)--cycle;

\strand[olive, thick] (3.25,-2) to [out=right, in=right] (3.25,-2.25)--(-0.25,-2.25) to [out=left, in=left] (-0.25,-2)--cycle;

\strand[purple, thick] (2.25,-2.75) to [out=right, in=right] (2.25,-3)--(-0.25,-3) to [out=left, in=left] (-0.25,-2.75)--cycle;

\flipcrossings{7,15,23,1,17,15,25,3,11,27,5,13,21}
\end{knot}
\filldraw (0,0) circle (1.5pt);
\filldraw (1,0) circle (1.5pt);
\filldraw (2,0) circle (1.5pt);
\filldraw (3,0) circle (1.5pt);
\draw (3.25,-0.5) node[right] {$-1$};
\draw (3.25,-1.25) node[right] {$-1$};
\draw (3.25,-2) node[right] {$-1$};
\draw (-0.25,-2.75) node[left] {$-1$};
\draw (4.25,-1) node {$\cong$};
\end{scope}

\begin{scope}[shift={(3.5,-7.6)}]
\begin{knot}[
	clip width=5,
	clip radius = 2pt,
	end tolerance = 1pt,
]

\strand (1,0) ellipse (0.7 and 0.35);
\strand (2,0) ellipse (0.7 and 0.35);
\strand (3,0) ellipse (0.7 and 0.35);
\strand (4,0) ellipse (0.7 and 0.35);
\strand[purple] (3,-0.35) ellipse (0.1 and 0.25);
\flipcrossings{1,3,5,7}
\end{knot}
\draw (1,0.35) node[above] {$\langle 2 \rangle$};
\draw (2,0.35) node[above] {$\langle 2 \rangle$};
\draw (3,0.35) node[above] {$\langle 1 \rangle$};
\draw (4,0.35) node[above] {$\langle 3 \rangle$};
\draw (3,-0.6) node[below] {$-1$};
\draw (2.5,-1.25) node {$\cup \phantom{0}(S^1\times D^3)$};
\end{scope}
\end{tikzpicture}
\end{center}
\caption{PALF on $W_{9,2}((2, 2, 1, 3))$.}
\label{W92}
\end{figure}

M.~Bhupal and B.~Ozbagci~\cite{BOz} also showed that the monodromy factorization for each minimal symplectic filling of the lens space $L(n,q)$ can be obtained by a sequence of monodromy substitutions that can be interpreted as a sequence of rational blowdowns from the minimal resolution of the corresponding singularity. 

\section{Lefschetz fibrations on minimal resolutions}

 In this section, as a first step towards proving our main theorem (Theorem~\ref{mainthm}), we construct a genus-$0$ or genus-$1$ positive allowable Lefschetz fibarion (PALF) on each minimal resolution of non-cyclic quotient surface singularities. 
 Note that a genus of the PALF is determined only by the existence of a bad vertex in the minimal resolution graph of the corresponding singularity. That is, a genus of the PALF is $0$ if the minimal resolution graph has no bad vertex, and a genus is $1$ otherwise.  
We subsequently check that a contact structure on the boundary induced from the PALF is the Milnor fillable contact structure 
so that every PALF obtained via monodromy substitutions is also a Stein filling of $(L, \xi_{\textrm{st}})$.

\subsection{No bad vertex cases}
If the minimal resolution graph $\Gamma$ of a quotient surface singularity does not have a bad vertex, then there is a well-known genus-$0$ PALF on the minimal resolution $\Gamma$, as demonstrated by D.~Gay and T.~Mark~\cite{GaM}. 
We consider the $2$-sphere $\Sigma_{i}$ with $b_i$ holes for each vertex $v_i$ with degree $-b_i$. Then the fiber surface $\Sigma$ is obtained by gluing $\Sigma_{i}$ along their boundaries according to $\Gamma$ and the vanishing cycles are the set of curves parallel to the boundary of each $\Sigma_{i}$. Note that we end up with only one right-handed Dehn twist on the connecting neck. We refer to Figure~\ref{genus0} below for an example.
Note that this PALF is compatible with the symplectic structure $\omega$ given by a convex plumbing $X_{\Gamma}$ of symplectic surfaces, where each vertex represents a symplectic surface with self-intersection $-b_i$ that intersect each other $\omega$-orthogonally according to $\Gamma$. 
Thus, the induced contact structure $\xi$ on the boundary $\partial X_{\Gamma}$ is compatible with the open book decomposition coming from the aforementioned PALF. H.~Park and A.~Stipsicz~\cite{PS} showed that $\xi$ is indeed the Milnor fillable contact structure. In fact, their argument holds for any negative-definite intersection matrix of $\Gamma$. 

\begin{figure}[htb]
\begin{center}
\begin{tikzpicture}[scale=0.6]
\begin{scope}
\draw[rounded corners] (-0.5,0)--(1.5,0) (-0.5,1)--(1.5,1) ;
\draw[<->](1.8,0.5)--(2.6,0.5);
\draw[rounded corners] (2.9,0)--(5.9,0) (2.9,1)--(3.9,1)--(3.9,2) (4.9,2)--(4.9,1)--(5.9,1);
\draw[rounded corners]  (7.3, 0)--(10.3,0) (7.3,1)--(8.3,1)--(8.3,2) (9.3,2)--(9.3,1)--(10.3,1);
\draw[rounded corners] (3.9,3.4)--(3.9,5.4) (4.9,3.4)--(4.9,5.4);
\draw[<->](1.8,0.5)--(2.6,0.5);
\draw[<->](6.2,0.5)--(7,0.5);
\draw[<->](4.4,2.3)--(4.4,3.1);

\draw (-0.50,0) arc (-90:90:0.2 and 0.5);
\draw (-0.50,0) arc (-90:-270:0.2 and 0.5);

\draw[red] (0,0) arc (-90:90:0.2 and 0.5);
\draw[dashed, red] (0,0) arc (-90:-270:0.2 and 0.5);

\draw (1.5,0) arc (-90:90:0.2 and 0.5);
\draw[dashed] (1.5,0) arc (-90:-270:0.2 and 0.5);

\draw[red] (1,0) arc (-90:90:0.2 and 0.5);
\draw[dashed, red] (1,0) arc (-90:-270:0.2 and 0.5);

\draw (2.9,0) arc (-90:-270:0.2 and 0.5);
\draw (2.9,0) arc (-90:90:0.2 and 0.5);

\draw[red] (3.4,0) arc (-90:90:0.2 and 0.5);
\draw[dashed, red] (3.4,0) arc (-90:-270:0.2 and 0.5);

\draw (5.9,0) arc (-90:90:0.2 and 0.5);
\draw[dashed] (5.9,0) arc (-90:-270:0.2 and 0.5);

\draw[red] (5.4,0) arc (-90:90:0.2 and 0.5);
\draw[dashed, red] (5.4,0) arc (-90:-270:0.2 and 0.5);

\draw (7.3,0) arc (-90:-270:0.2 and 0.5);
\draw (7.3,0) arc (-90:90:0.2 and 0.5);

\draw[red] (7.8,0) arc (-90:90:0.2 and 0.5);
\draw[dashed, red] (7.8,0) arc (-90:-270:0.2 and 0.5);

\draw (10.3,0) arc (-90:90:0.2 and 0.5);
\draw[dashed] (10.3,0) arc (-90:-270:0.2 and 0.5);

\draw[red] (9.8,0) arc (-90:90:0.2 and 0.5);
\draw[dashed, red] (9.8,0) arc (-90:-270:0.2 and 0.5);

\draw (3.9,2) arc (180:360:0.5 and 0.2);
\draw (3.9,2) arc (-180:-360:0.5 and 0.2);

\draw[red] (3.9,1.5) arc (180:360:0.5 and 0.2);
\draw[red,dashed] (3.9,1.5) arc (-180:-360:0.5 and 0.2);

\draw (3.9,3.4) arc (180:360:0.5 and 0.2);
\draw[dashed] (3.9,3.4) arc (-180:-360:0.5 and 0.2);

\draw[red] (3.9,3.9) arc (180:360:0.5 and 0.2);
\draw[red,dashed] (3.9,3.9) arc (-180:-360:0.5 and 0.2);

\draw (3.9,5.4) arc (180:360:0.5 and 0.2);
\draw (3.9,5.4) arc (-180:-360:0.5 and 0.2);

\draw[red] (3.9,4.9) arc (180:360:0.5 and 0.2);
\draw[red,dashed] (3.9,4.9) arc (-180:-360:0.5 and 0.2);

\draw (8.3,2) arc (180:360:0.5 and 0.2);
\draw (8.3,2) arc (-180:-360:0.5 and 0.2);

\draw[red] (8.3,1.5) arc (180:360:0.5 and 0.2);
\draw[red,dashed] (8.3,1.5) arc (-180:-360:0.5 and 0.2);

\draw[->,very thick] (11,1.5)--(12,1.5);

\end{scope}
\begin{scope}[shift={(13,0)}]
\draw[rounded corners] (0,0)--(6,0) (0,1)--(2,1)--(2,3) (3,3)--(3,1)--(4,1)--(4,2) (5,2)--(5,1)--(6,1);
\draw (0,0) arc (-90:90:0.2 and 0.5);
\draw (0,0) arc (-90:-270:0.2 and 0.5);

\draw[red] (0.5,0) arc (-90:90:0.2 and 0.5);
\draw[red,dashed] (0.5,0) arc (-90:-270:0.2 and 0.5);

\draw[red] (1.5,0) arc (-90:90:0.2 and 0.5);
\draw[red,dashed] (1.5,0) arc (-90:-270:0.2 and 0.5);

\draw (6,0) arc (-90:90:0.2 and 0.5);
\draw[dashed] (6,0) arc (-90:-270:0.2 and 0.5);

\draw[red] (5.5,0) arc (-90:90:0.2 and 0.5);
\draw[red,dashed] (5.5,0) arc (-90:-270:0.2 and 0.5);

\draw[red] (3.5,0) arc (-90:90:0.2 and 0.5);
\draw[red,dashed] (3.5,0) arc (-90:-270:0.2 and 0.5);

\draw (2,3) arc (180:360:0.5 and 0.2);
\draw (2,3) arc (-180:-360:0.5 and 0.2);

\draw[red] (2,2.5) arc (180:360:0.5 and 0.2);
\draw[red,dashed] (2,2.5) arc (-180:-360:0.5 and 0.2);
\draw[red] (2,2.5) arc (180:360:0.5 and 0.2);
\draw[red,dashed] (2,2.5) arc (-180:-360:0.5 and 0.2);
\draw[red] (2,1.5) arc (180:360:0.5 and 0.2);
\draw[red,dashed] (2,1.5) arc (-180:-360:0.5 and 0.2);

\draw (4,2) arc (180:360:0.5 and 0.2);
\draw (4,2) arc (-180:-360:0.5 and 0.2);

\draw[red] (4,1.5) arc (180:360:0.5 and 0.2);
\draw[red,dashed] (4,1.5) arc (-180:-360:0.5 and 0.2);

\end{scope}
\begin{scope}[shift={(0,-5)}]
\draw (5.4,0) ellipse (5 and 2);
\draw (5.4,0) circle (0.5);
\draw[red] (5.4,0) circle (0.65);
\draw[red] (5.4,0) circle (0.8);

\draw (3.2,0) circle (0.5);
\draw[red] (3.2,0) circle (0.65);
\draw[red] (3.2,0) circle (0.8);

\draw (7.6,0) circle (0.5);
\draw[red] (7.6,0) circle (0.65);

\draw[red] (4.3,0) ellipse (2.5 and 1.2);
\draw[red] (5.4,0) ellipse (4.5 and 1.7);
\draw[->,very thick] (11,0)--(12,0);
\begin{knot}[
	clip width=5,
	clip radius = 2pt,
	end tolerance = 1pt,
]
\filldraw (15,2.25) circle (2pt);
\filldraw (16.5,2.25) circle (2pt);
\filldraw (18,2.25) circle (2pt);
\strand (15,3)--(15,-3);
\strand (16.5,3)--(16.5,-3);
\strand (18,3)--(18,-3);
\strand[red](15,1) ellipse (0.5 and 0.2);
\strand[red](15,0) ellipse (0.5 and 0.2);
\strand[red](16.5,1) ellipse (0.5 and 0.2);
\strand[red](16.5,0) ellipse (0.5 and 0.2);

\strand[red](18,0.5) ellipse (0.5 and 0.2);

\strand[red](15.75,-1) ellipse (1.5 and 0.25);

\strand[red](16.5,-2) ellipse (2.5 and 0.25);

\flipcrossings{2,10,4,12,18,6,14,8,16,20}
\draw (14.5,1) node[left] {$-1$};
\draw (14.5,0) node[left] {$-1$};
\draw (14.25,-1) node[left] {$-1$};
\draw (14,-2) node[left] {$-1$};

\draw (18.5,0.5) node[right] {$-1$};
\draw (17,-0.1) node[right] {$-1$};
\draw (17,1.1) node[right] {$-1$};

\end{knot}
\end{scope}
\end{tikzpicture}
\caption{A genus-$0$ PALF on minimal resolution of $D_{8,3}$.}
\label{genus0}
\end{center}
\end{figure}
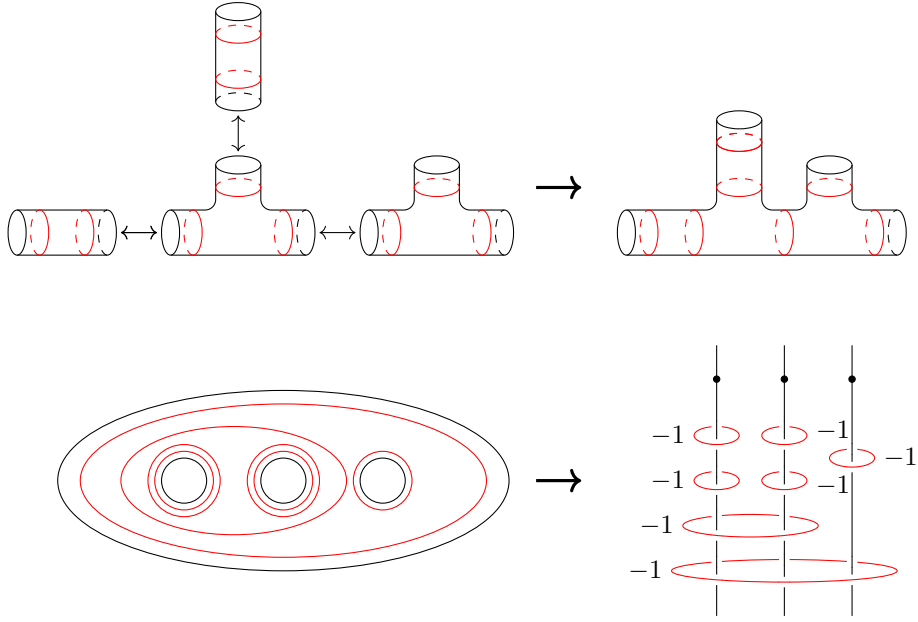

\subsection{Bad vertex cases}
If the minimal resolution graph $\Gamma$ of a non-cyclic quotient surface singularity has a bad vertex, we now construct a genus-$1$ PALF on the minimal resolution $\Gamma$ as follows: 
First we construct a PALF on $X_L$, where $X_L$ is the minimal resolution of a cyclic singularity determined by 
 a maximal linear subgraph $\Gamma_L$ of $\Gamma$. We consider a $4$-dimensional Kirby diagram of $X_L$, which can be easily obtained from the PALF of $X_L$. We could subsequently obtain a Kirby diagram of $X_{\Gamma}$ by adding a $2$-handle $h$ or two $2$-handles $\{h_1, h_2\}$ to that of $X_L$, depending on which type of arm is not in $\Gamma_L$. After introducing a cancelling $1$-handle/$2$-handle pair, the $2$-handles not coming from the Kirby diagram of $X_L$ can be thought of as vanishing cycles of a new fiber $F$, which is obtained by attaching a $1$-handle to the surface $F_L$. Note that the new fiber $F$ is a genus-$1$ surface with holes. We refer to Figure~\ref{genus1palf} below for an example.

\begin{figure}[htbp]
\begin{center}
\begin{tikzpicture}[scale=0.75]
\begin{scope}

\begin{knot}[
	clip width=5,
	clip radius = 2pt,
	end tolerance = 1pt,
]

\draw (0,0) -- (4,0);
\draw[dashed] (0,0.4) -- (4,0.4);
\draw[dashed] (0,-0.4) -- (4,-0.4);
\draw [dashed] (0,0.4) arc (90:270:0.4 and 0.4);
\draw [dashed] (4,-0.4) arc (-90:90:0.4 and 0.4);
\draw (2,0) -- (2,2);
\filldraw (2,0) circle (2pt) ;
\filldraw (0,0) circle (2pt) ;
\filldraw (4,0) circle (2pt) ;
\filldraw (2,2) circle (2pt) ;
\node at (0,-0.8) {$-2$};
\node at (2,-0.8) {$-2$};
\node at (4,-0.8) {$-3$};
\node at (2.5,2) {$-2$};
\node at (4,0.8) {$\Gamma_L$};
\draw [thick,->] (6,0) -- (7,0);

\strand (10,4) -- (10,-2.5); \strand (12,4) -- (12,-2.5);
\filldraw (10,3.5) circle (2pt);\filldraw (12,3.5) circle (2pt);
\strand (10,2.4) ellipse (0.8 and 0.3);
\strand (10,1.2) ellipse (0.8 and 0.3);
\strand (10,0) ellipse (0.8 and 0.3);
\strand (12,1.2) ellipse (0.8 and 0.3);
\strand (11,-1.3) ellipse (1.8 and 0.4);
\flipcrossings {8,12,2,4,6,10}
\end{knot}
\node at (10-0.8-0.5,2.4) {$-1$};
\node at (10-0.8-0.5,1.2) {$-1$};
\node at (10-0.8-0.5,0) {$-1$};
\node at (12+0.8+0.5,1.2) {$-1$};
\node at (10-0.8-0.5,-1.3) {$-1$};
\end{scope}

\begin{scope}[shift={(0,-9)}]

\begin{knot}[
	clip width=5,
	clip radius = 2pt,
	end tolerance = 1pt,
]

\draw (0,0) -- (4,0);
\draw (2,0) -- (2,2);
\filldraw (2,0) circle (2pt) ;
\filldraw (0,0) circle (2pt) ;
\filldraw (4,0) circle (2pt) ;
\filldraw (2,2) circle (2pt) ;
\node at (0,-0.5) {$-2$};
\node at (2,-0.5) {$-2$};
\node at (4,-0.5) {$-3$};
\node at (2.5,2) {$-2$};
\draw [thick,->] (6,0) -- (7,0);

\strand (10+0.1,3+0.2)..controls +(0.3,0) and +(0,-0.3) .. (10+0.8,3+0.3+0.2) .. controls +(0,0.3) and +(-0.3,0) .. (10+0.1,3+0.6+0.2) .. controls +(+0.3,0) and +(0,-0.3) .. (10-0.8,3+0.9+0.2) .. controls +(0,+0.3) and +(-0.3,0) .. (10+0.1,3+1.2+0.2) .. controls +(+0.3,0) and +(0,-0.3) .. (10+0.8,3+1.5+0.2) .. controls +(0,+0.3) and +(-0.3,0) .. (10+0.1,3+1.8+0.2) .. controls +(-1,0) and +(0,1) .. (10-1.6,3+0.9+0.2) .. controls +(0,-1) and +(-1,0) .. (10+0.1,3+0.2);
\strand (10,6) -- (10,-2.5); \strand (12,6) -- (12,-2.5);
\filldraw (10,5.5) circle (2pt);\filldraw (12,5.5) circle (2pt);
\strand (10,2.4) ellipse (0.8 and 0.3);
\strand (10,1.2) ellipse (0.8 and 0.3);
\strand (10,0) ellipse (0.8 and 0.3);
\strand (12,1.2) ellipse (0.8 and 0.3);
\strand (11,-1.3) ellipse (1.8 and 0.4);
\flipcrossings {8,12,1,3,6,10,14,16}
\end{knot}
\node at (10-0.8-0.5,2.4) {$-1$};
\node at (10-0.8-0.5,1.2) {$-1$};
\node at (10-0.8-0.5,0) {$-1$};
\node at (12+0.8+0.5,1.2) {$-1$};
\node at (10-0.8-0.5,-1.3) {$-1$};
\node at (10-1.6-0.5,3+0.9+0.2) {$0$};
\end{scope}

\begin{scope}

\begin{knot}[
	clip width=5,
	clip radius = 2pt,
	end tolerance = 1pt,
]

\draw (0,0) -- (4,0);
\draw[dashed] (0,0.4) -- (4,0.4);
\draw[dashed] (0,-0.4) -- (4,-0.4);
\draw [dashed] (0,0.4) arc (90:270:0.4 and 0.4);
\draw [dashed] (4,-0.4) arc (-90:90:0.4 and 0.4);
\draw (2,0) -- (2,2);
\filldraw (2,0) circle (2pt) ;
\filldraw (0,0) circle (2pt) ;
\filldraw (4,0) circle (2pt) ;
\filldraw (2,2) circle (2pt) ;
\node at (0,-0.8) {$-2$};
\node at (2,-0.8) {$-2$};
\node at (4,-0.8) {$-3$};
\node at (2.5,2) {$-2$};
\node at (4,0.8) {$\Gamma_L$};
\draw [thick,->] (6,0) -- (7,0);

\strand (10,4) -- (10,-2.5); \strand (12,4) -- (12,-2.5);
\filldraw (10,3.5) circle (2pt);\filldraw (12,3.5) circle (2pt);
\strand (10,2.4) ellipse (0.8 and 0.3);
\strand (10,1.2) ellipse (0.8 and 0.3);
\strand (10,0) ellipse (0.8 and 0.3);
\strand (12,1.2) ellipse (0.8 and 0.3);
\strand (11,-1.3) ellipse (1.8 and 0.4);
\flipcrossings {8,12,2,4,6,10}
\end{knot}
\node at (10-0.8-0.5,2.4) {$-1$};
\node at (10-0.8-0.5,1.2) {$-1$};
\node at (10-0.8-0.5,0) {$-1$};
\node at (12+0.8+0.5,1.2) {$-1$};
\node at (10-0.8-0.5,-1.3) {$-1$};
\end{scope}

\begin{scope}[shift={(-3.5,-18)}]

\begin{knot}[
	clip width=5,
	clip radius = 2pt,
	end tolerance = 1pt,
]

\draw [->] (14,6) -- (13,5);
\strand[blue] (10+0.1,3+0.2)..controls +(0.3,0) and +(0,-0.3) .. (10+0.8,3+0.3+0.2) .. controls +(0,0.3) and +(-0.3,0) .. (10+0.1,3+0.6+0.2) .. controls +(+0.3,0) and +(0,-0.3) .. (10-0.8,3+0.9+0.2) .. controls +(0,+0.3) and +(-0.3,0) .. (10+0.1,3+1.2+0.2) .. controls +(+0.3,0) and +(0,-0.3) .. (10+0.8,3+1.5+0.2) .. controls +(0,+0.3) and +(-0.3,0) .. (10+0.1,3+1.8+0.2) .. controls +(-1,0) and +(0,1) .. (10-3.5,3+0.9+0.2) .. controls +(0,-1) and +(-1,0) .. (10+0.1,3+0.2);
\strand (10,6) -- (10,-2.5); \strand (12,6) -- (12,-2.5);
\strand (8,6) -- (8,-2.5);
\filldraw (10,5.5) circle (2pt);\filldraw (12,5.5) circle (2pt);
\filldraw (8,5.5) circle (2pt);

\strand[red] (8,3+1.05) ellipse (0.8 and 0.3);

\strand[orange] (10,2.4) ellipse (0.8 and 0.3);
\strand[orange] (10,1.2) ellipse (0.8 and 0.3);
\strand (10,0) ellipse (0.8 and 0.3);
\strand (12,1.2) ellipse (0.8 and 0.3);
\strand (11,-1.3) ellipse (1.8 and 0.4);
\flipcrossings {8,12,1,3,10,14,16,20,5,18}
\end{knot}

\node at (8-1.1,3+1.05) {$-1$};

\node at (10-0.8-0.5,2.4) {$-1$};
\node at (10-0.8-0.5,1.2) {$-1$};
\node at (10-0.8-0.5,0) {$-1$};
\node at (12+0.8+0.5,1.2) {$-1$};
\node at (10-0.8-0.5,-1.3) {$-1$};
\node at (9,5.3) {$+1$};
\end{scope}
\end{tikzpicture}

\end{center}
\end{figure}
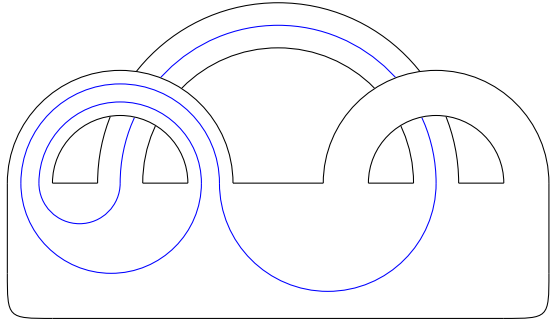
\begin{figure}[htbp]
\begin{center}

\begin{tikzpicture}[scale=0.6]

\draw (10,0) arc (0:180:4 and 4);
\draw (9,0) arc (0:180:3 and 3);
\draw[blue] (9.5,0) arc (0:180:3.5 and 3.5);
\filldraw[fill=white,white] (0,0)--(1,0) (1,0) arc (180:0:1.5 and 1.5) (4,0)--(5,0) (5,0) arc (0:180:2.5 and 2.5);
\filldraw[fill=white,white] (7,0)--(1+7,0) (1+7,0) arc (180:0:1.5 and 1.5) (4+7,0)--(5+7,0) (5+7,0) arc (0:180:2.5 and 2.5);
\draw (1,0)--(2,0) (3,0)--(4,0) (5,0)--(7,0) (8,0)--(9,0) (10,0)--(11,0);
\draw (5,0) arc (0:180:2.5 and 2.5);
\draw (4,0) arc (0:180:1.5 and 1.5);
\draw (5+7,0) arc (0:180:2.5 and 2.5);
\draw (4+7,0) arc (0:180:1.5 and 1.5);

\draw[blue] (5-0.3,0) arc (0:180:2.2 and 2.2);
\draw[blue] (4+0.3,0) arc (0:180:1.8 and 1.8);
\draw[blue] (2.5,0) arc (0:-180:0.9 and 0.9);
\draw[blue] (4.3,0) arc (0:-180:2 and 2);
\draw[blue] (9.5,0) arc (0:-180:2.4 and 2.4);

\draw (0,0) -- (0,-2) (1,-3) -- (11,-3) (12,-2) -- (12,0);
\draw (0,-2)..controls +(0,-1) and +(-1,0) ..(1,-3);
\draw (11,-3)..controls +(1,0) and +(0,-1) ..(12,-2);

\end{tikzpicture}

\end{center}
\caption{A genus-1 PALF on the minimal resolution of  $D_{5,3}$.}
\label{genus1palf}
\end{figure}

Next, we check that a contact structure on the boundary induced from the PALF constructed above is the Milnor fillable contact structure. 
First, recall that, the $2$-plane field $\xi$ induces a $Spin^c$ structure $t_{\xi}$ on $L$ for a contact $3$-manifold $(L,\xi)$. Furthermore, if $(W,J)$ is a Stein filling of $(L,\xi)$, then $t_{\xi}$ is a restriction of $Spin^c$ structure $\mathcal{S}$ on $W$ to $\partial W=L$ induced by its complex structure $J$ on $W$. On the other hand, there is a theorem of Gay-Stipsicz~\cite{GaS} that characterizes the contact structure on the link of a quotient surface singularity.

\begin{thm}[\cite{GaS}]
\label{GaS}
Suppose that a small Seifert $3$-manifold $M\! =\! M(s_0;r_1,r_2,r_3)$ satisfies $s_0 \leq -2$ and $M$ is an $L$-space. Then two tight contact structures $\xi_1,\xi_2$ on $M$ are isotopic if and only if $t_{\xi_1}=t_{\xi_2}$.
\end{thm}

\begin{thm}
The contact structure on the link of non-cyclic quotient surface singularities induced by the PALF constructed above is Milnor fillable. 
\end{thm}
\begin{proof}
First note that, since a convex plumbing $X_{\Gamma}$ of the minimal resolution graph $\Gamma$ of a quotient singularity is simply connected, the $Spin^c$ structure $\mathcal{S}$ on $X_{\Gamma}$ is determined by the first Chern class $c_1(\mathcal{S})$. 
On the other hand, $t_{\xi_{\mathrm{st}}}$ is a restriction of $\mathcal{S}$ whose first Chern class $c_1(\mathcal{S})$ satisfies the adjunction equality on each vertex in $\Gamma$. 
Hence, according to Theorem~\ref{GaS} above, a PALF on $X_{\Gamma}$ induces the Milnor fillable contact structure on the boundary if and only if $c_1(J)$ satisfies the adjunction equality for each vertex in $\Gamma$, where $J$ is a complex structure coming from the Stein structure of the PALF.
From the PALF on $X_{\Gamma}$ constructed above, we can compute the first Chern class  $c_1(J)$ in terms of vanishing cycles $C_i$: $c_1(J)$ is represented by a co-cycle whose value on the $2$-handle corresponding to $C_i$ is the rotation number $r(C_i)$~\cite{G}, which can be computed once we fix a trivialization of the tangent bundle of a page~\cite{EtOz2}.
The vertices in $\Gamma_L$ satisfy the adjunction equality because the PALF for the no bad vertex cases induces the Milnor fillable contact structure \cite{PS}.
The homology classes of the vertices not in $\Gamma_L$ can be represented by new vanishing cycles together with some vanishing cycles in $\Gamma_L$, thus we can check whether they satisfy the adjunction equality by computing the rotation number of the vanishing cycles. Note that all non-cyclic quotient singularities can be divided into the following three cases: Dihedral singularities, singularities of type $(3, 1)$, and singularities of type $(3, 2)$. \vspace{0.5 ex}

\emph{Dihedral singularities and singularities of type $(3,2)$}:
There is only one degree $-2$ vertex $v$ which is not in $\Gamma_L$ for dihedral cases. If we construct a genus-$1$ PALF on the minimal resolution as shown above, then the global monodromy of the PALF should contain $C_{\textrm{red}}C_{\textrm{blue}}C_{\textrm{orange}}^2$ as a subword and $v$ is homologically equal to $C_{\textrm{blue}}-C_{\textrm{red}}+2C_{\textrm{orange}}$. Once we fix a trivialization of the tangent bundle of fiber $F$ as a natural extension of a trivialization of the tangent bundle of $\mathbb{R}^2$, the rotation numbers of blue, red, and orange vanishing cycles are $-1$, $+1$, and $+1$, respectively. This means that the $-2$ vertex $v$ satisfies the adjunction equality. The only difference between the dihedral case and the type $(3, 2)$ case is that there is another degree $-2$ vertex $v'$ connected to $v$ which is not in $\Gamma_L$ for the $(3, 2)$ case. Hence, for type $(3, 2)$ singularities, the global monodromy of the PALF on the minimal resolution should contain $C_{\textrm{red}}C_{\textrm{blue}}^2C_{\textrm{orange}}^2$ as a subword. The vertex $v'$ is also homologically $C_{\textrm{blue}}-C_{\textrm{blue}}$, meaning every vertex in $\Gamma$ satisfies the adjunction equality. See Figure ~\ref{genus1dihedral}.\vspace{0.5 em}

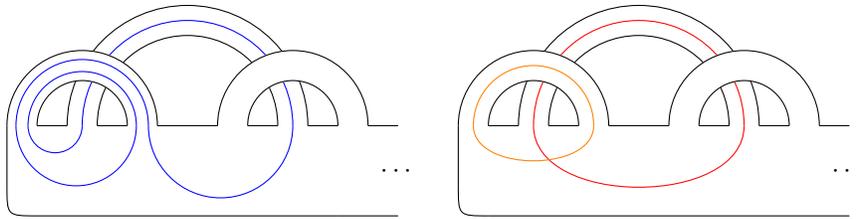
\begin{figure}[htbp]
\begin{center}

\begin{tikzpicture}[scale=0.4]

\begin{scope}
\draw (10,0) arc (0:180:4 and 4);
\draw (9,0) arc (0:180:3 and 3);
\draw[blue] (9.5,0) arc (0:180:3.5 and 3.5);
\filldraw[fill=white,white] (0,0)--(1,0) (1,0) arc (180:0:1.5 and 1.5) (4,0)--(5,0) (5,0) arc (0:180:2.5 and 2.5);
\filldraw[fill=white,white] (7,0)--(1+7,0) (1+7,0) arc (180:0:1.5 and 1.5) (4+7,0)--(5+7,0) (5+7,0) arc (0:180:2.5 and 2.5);
\draw (1,0)--(2,0) (3,0)--(4,0) (5,0)--(7,0) (8,0)--(9,0) (10,0)--(11,0);
\draw (5,0) arc (0:180:2.5 and 2.5);
\draw (4,0) arc (0:180:1.5 and 1.5);
\draw (5+7,0) arc (0:180:2.5 and 2.5);
\draw (4+7,0) arc (0:180:1.5 and 1.5);

\draw[blue] (5-0.3,0) arc (0:180:2.2 and 2.2);
\draw[blue] (4+0.3,0) arc (0:180:1.8 and 1.8);
\draw[blue] (2.5,0) arc (0:-180:0.9 and 0.9);
\draw[blue] (4.3,0) arc (0:-180:2 and 2);
\draw[blue] (9.5,0) arc (0:-180:2.4 and 2.4);

\draw (0,0) -- (0,-2) (1,-3) -- (11,-3) ;
\draw (0,-2)..controls +(0,-1) and +(-1,0) ..(1,-3);
\draw (11,-3)--(13,-3);
\draw (12,0)--(13,0);
\draw (13,-1.5) node {$\cdots$};
\end{scope}
\begin{scope}[shift={(15,0)}]
\draw (10,0) arc (0:180:4 and 4);
\draw (9,0) arc (0:180:3 and 3);
\draw[red] (9.5,0) arc (0:180:3.5 and 3.5);
\draw[red] (9.5,0) to [out=down, in=down] (2.5,0);
\filldraw[fill=white,white] (0,0)--(1,0) (1,0) arc (180:0:1.5 and 1.5) (4,0)--(5,0) (5,0) arc (0:180:2.5 and 2.5);
\filldraw[fill=white,white] (7,0)--(1+7,0) (1+7,0) arc (180:0:1.5 and 1.5) (4+7,0)--(5+7,0) (5+7,0) arc (0:180:2.5 and 2.5);
\draw (1,0)--(2,0) (3,0)--(4,0) (5,0)--(7,0) (8,0)--(9,0) (10,0)--(11,0);
\draw (5,0) arc (0:180:2.5 and 2.5);
\draw (4,0) arc (0:180:1.5 and 1.5);
\draw (5+7,0) arc (0:180:2.5 and 2.5);
\draw (4+7,0) arc (0:180:1.5 and 1.5);
\draw[orange] (4.5,0) arc (0:180:2 and 2);
\draw[orange] (4.5,0) to [out=down, in=down] (0.5,0);


\draw (0,0) -- (0,-2) (1,-3) -- (11,-3) ;
\draw (0,-2)..controls +(0,-1) and +(-1,0) ..(1,-3);
\draw (11,-3)--(13,-3);
\draw (12,0)--(13,0);
\draw (13,-1.5) node {$\cdots$};

\end{scope}
\end{tikzpicture}

\end{center}
\caption{A genus-$1$ PALF  on the minimal resolution of dihedral singularities.}
\label{genus1dihedral}
\end{figure}

\emph{Singularities of type (3,1)}:
Consider a maximal linear subgraph $\Gamma_L$ starting with a degree $-3$ vertex $v$. Since $\Gamma_L$ starts with a degree $-3$ vertex $v$, the global monodromy of the PALF should contain $C_{\textrm{red}}C_{\textrm{blue}}C_{\textrm{olive}}C_{\textrm{orange}}$ as a subword. A similar argument shows that the genus-$1$ PALF we constructed induces a Milnor fillable contact structure on the boundary. See Figure~\ref{genus1other}.


\begin{figure}[htbp]
\begin{tikzpicture}[scale=0.4]

\draw (0,0)--(-1,0);
\draw (-1,0) arc (0:180:2.5 and 2.5);
\draw (-2,0) arc (0:180:1.5 and 1.5);
\draw (-2,0)--(-5,0);

\draw(12,0)--(14,0);


\draw (10,0) arc (0:180:4 and 4);
\draw (9,0) arc (0:180:3 and 3);
\draw[blue] (9.5,0) arc (0:180:3.5 and 3.5);
\filldraw[fill=white,white] (0,0)--(1.5,0) (1.5,0) arc (180:0:1 and 1) (3.5,0)--(5,0) (5,0) arc (0:180:2.5 and 2.5);
\filldraw[fill=white,white] (7,0)--(1+7,0) (1+7,0) arc (180:0:1.5 and 1.5) (4+7,0)--(5+7,0) (5+7,0) arc (0:180:2.5 and 2.5);
\draw (1.5,0)--(2,0) (3,0)--(3.5,0) (5,0)--(7,0) (8,0)--(9,0) (10,0)--(11,0);
\draw (5,0) arc (0:180:2.5 and 2.5);
\draw (3.5,0) arc (0:180:1 and 1);
\draw (5+7,0) arc (0:180:2.5 and 2.5);
\draw (4+7,0) arc (0:180:1.5 and 1.5);

\draw[blue] (4.2,0) arc (0:180:1.7 and 1.7);
\draw[blue] (3.8,0) arc (0:180:1.3 and 1.3);
\draw[blue] (2.5,0) arc (0:-180:0.65 and 0.65);
\draw[blue] (9.5,0) to [out=down, in=down] (4.2,0);
\draw[blue] (-1.7,0) arc (0:180:1.8);
\draw[blue] (-5.3,0) to [out=down, in=left] (-0.6,-2) to [out=right, in=down] (3.8,0);
\draw[blue] (-1.7,0) to [out=down, in=left] (-0.45,-1) to [out=right, in=down]  (0.8,0);

\draw  (-6,0)--(-6,-2) (-5,-3) -- (14,-3);
\draw (-6,-2)..controls +(0,-1) and +(-1,0) ..(-5,-3);
\draw (14,-1.5) node {$\cdots$};
\end{tikzpicture}

\vspace{1 em}

\begin{tikzpicture}[scale=0.4]

\draw (0,0)--(-1,0);
\draw (-1,0) arc (0:180:2.5 and 2.5);
\draw (-2,0) arc (0:180:1.5 and 1.5);
\draw (-2,0)--(-5,0);

\draw(12,0)--(14,0);



\draw (10,0) arc (0:180:4 and 4);
\draw (9,0) arc (0:180:3 and 3);
\draw[red] (9.5,0) arc (0:180:3.5 and 3.5);
\draw[red] (9.5,0) to [out=down, in=down] (2.5,0);
\filldraw[fill=white,white] (0,0)--(1,0) (1,0) arc (180:0:1.5 and 1.5) (4,0)--(5,0) (5,0) arc (0:180:2.5 and 2.5);
\filldraw[fill=white,white] (7,0)--(1+7,0) (1+7,0) arc (180:0:1.5 and 1.5) (4+7,0)--(5+7,0) (5+7,0) arc (0:180:2.5 and 2.5);
\draw (1,0)--(2,0) (3,0)--(4,0) (5,0)--(7,0) (8,0)--(9,0) (10,0)--(11,0);
\draw (5,0) arc (0:180:2.5 and 2.5);
\draw[olive] (4.7,0) arc (0:180:2.2);
\draw[olive] (-1.3,0) arc (0:180:2.2);
\draw (4,0) arc (0:180:1.5 and 1.5);
\draw (5+7,0) arc (0:180:2.5 and 2.5);
\draw (4+7,0) arc (0:180:1.5 and 1.5);

\draw[olive] (4.7,0) to [out=down,in=right] (0,-2) to [out=left,in=down] (-5.7,0);
\draw[olive] (0.3,0) to [out=down,in=down] (-1.3,0);

\draw[orange] (4.3,0) arc (0:180: 1.8);
\draw[orange] (4.3,0) to [out=down,in=down] (0.7,0);





\draw  (-6,0)--(-6,-2) (-5,-3) -- (14,-3);
\draw (-6,-2)..controls +(0,-1) and +(-1,0) ..(-5,-3);
\draw (14,-1.5) node {$\cdots$};
\end{tikzpicture}

\caption{A genus-$1$ PALF on the minimal resolution of type $(3,1)$ singularities.}
\label{genus1other}
\end{figure}
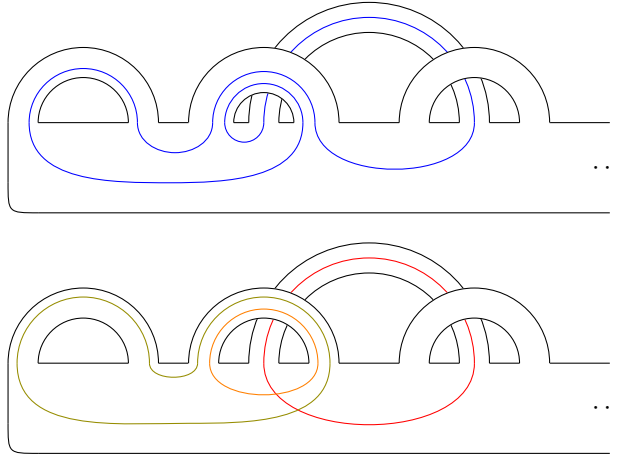


\end{proof}

%
%
%
\section{Lefschetz fibrations on minimal symplectic fillings}

As mentioned in the Introduction, M.~Bhupal and K.~Ono~\cite{BOn} listed all possible minimal symplectic fillings for non-cyclic quotient surface singularities. In fact, they showed that each minimal symplectic filling of a non-cyclic singularity $(X,0)$ is orientation-preserving diffeomorphic to $Z-\nu(E_{\infty})$, where $E_{\infty}$ is the \emph{compactifying divisor} of $X$ embedded in a rational symplectic $4$-manifold $Z$. They also found all possible pairs of $(Z, E_{\infty})$ for non-cyclic quotient  singularities. 
On the other hand, H.~Park, J.~Park, D.~Shin, and G.~Urz{\'u}a~\cite{PPSU} observed that the number of $P$-resolutions in J.~Stevens~\cite{Ste2} and that of minimal symplectic fillings in Bhupal-Ono's list~\cite{BOn} are nearly equal. Since there is a one-to-one correspondence between Milnor fibers and $P$-resolutions for quotient singularities~\cite{KSB}, it is natural to ask whether every minimal symplectic filling of non-cyclic quotient singularities is a Milnor fiber. 
This was proven in ~\cite{PPSU} using the corresponding complex model.
In fact, they even constructed an explicit one-to-one correspondence between the minimal symplectic fillings and the Milnor fibers of non-cyclic quotient surface singularities. 

What follows is a strategy for proving our main theorem (Theorem~\ref{mainthm}). 
For a given $P$-resolution $Y$ of a non-cyclic quotient singularity $X$, we construct a PALF on the minimal resolution of $X$, after appropriate monodromy substitutions, the resulting $4$-manifold is diffeomorphic to a $4$-manifold obtained by rationally blowing down all singularities of class $T$ in the minimal resolution $Z$ of $Y$. Since any Milnor fiber of quotient surface singularities can be obtained by rationally blowing down all singularities of class $T$ in the minimal resolution of the corresponding $P$-resolution topologically, we would be done.

Now we construct a PALF on each minimal symplectic filling of a non-cyclic quotient  singularity $X$ via the corresponding $P$-resolution $Y$. In other words, we first construct a PALF on the minimal resolution of $X$ and 
we find a suitable monodromy substitution to obtain a PALF on the $4$-manifold obtained by rationally blowing down all singularities of class $T$ in $Z$.
Let $\Gamma_Z$ be the dual graph of the minimal resolution $Z$ of $Y$. For the sake of convenience, we divide all $P$-resolutions into the following two cases: those with and without a maximal linear subgraph $\Gamma_L$ of $\Gamma_Z$ containing all singularities of class $T$.

\subsection{Case 1}
\label{case1} Let $Y$ be a $P$-resolution of a non-cyclic quotient singularity $X$ whose minimal resolution graph $\Gamma_Z$ has a maximal linear subgraph $\Gamma_L$ containing all singularities of class $T$ in $Y$.
Note that the subgraph $\Gamma_L$ becomes the minimal resolution graph of a $P$-resolution $Y'$ for some cyclic quotient singularity $X'$.
Then, by combining a PALF on the minimal resolution of $X'$ and a technique developed in Section~4, we can construct a PALF on the minimal resolution of $X$ based on the PALF of the minimal resolution of $X'$. 
Explicitly, 
starting from a PALF $(F_{X'}, y_1y_2\cdots y_m)$ 
on the minimal resolution of $X'$, 
we obtain a PALF ($F_X, x_1\cdots x_n\widetilde{y}_1\widetilde{y}_2\cdots\widetilde{y}_m)$ 
on the minimal resolution of $X$, where the vanishing cycles $\widetilde{y}_i$'s in $F_X$ are natural extensions of the corresponding vanishing cycles $y_i$'s in $F_{X'}$, and the vanishing cycles $\{x_1, \ldots, x_n\}$ come from the corresponding vertices in the arm which is not contained in $\Gamma_L$.
Note that a genus of the generic fiber $F_X$ depends on the existence of a bad vertex in the minimal resolution of $X$.
We subsequently find a monodromy substitution that yields a PALF on $Y$.
Since $Y'$ contains all singularities of class $T$ lying in $Y$, 
if a monodromy substitution of the form 
$y_1y_2\cdots y_m=z_1z_2\cdots z_l$ 
yields a PALF on $Y'$, then $(F_X, x_1\cdots x_n\widetilde{z}_1\widetilde{z}_2\cdots\widetilde{z}_l)$ 
yields a desired PALF on $Y$, where the vanishing cycles $\widetilde{z}_i$'s in $F_X$ are natural extensions of the corresponding vanishing cycles $z_i$'s in $F_{X'}$.

If the genus of generic fiber $F_X$ is $1$, then a monodromy substitution of the form $\widetilde{y}_1\widetilde{y}_2\cdots \widetilde{y}_m=\widetilde{z}_1\widetilde{z}_2\cdots \widetilde{z}_l$ still can be interpreted as a sequence of rational blowdowns.
If the genus of the generic fiber $F_X$ is $0$, then there could be a vertex whose degree is strictly less than $-2$ in the arm not contained in $\Gamma_L$. In this case, a monodromy substitution of the form $\widetilde{y}_1\widetilde{y}_2\cdots \widetilde{y}_m=\widetilde{z}_1\widetilde{z}_2\cdots \widetilde{z}_l$ in $F_X$ does not correspond to a sequence of rational blowdown surgeries because, in general, there could be more holes in $F_{X}$ enclosed by $\widetilde{y_i}$ and $\widetilde{z_j}$ than holes in $F_{X'}$ enclosed by $y_i$ and $z_j$ for some $\widetilde{y_i}$ and $\widetilde{z_j}$. The construction of the PALF on the minimal resolution shows that there is a vanishing cycle $x_i$ enclosing only the new hole for each new hole in $F_X$. Thus, after adding such $\{x_i\}$ to both sides, $x_{i_1}x_{i_2}\cdots x_{i_k}\widetilde{y}_1\widetilde{y}_2\cdots \widetilde{y}_m=x_{i_1}x_{i_2}\cdots x_{i_k}\widetilde{z}_1\widetilde{z}_2\cdots \widetilde{z}_l$ is a positive stabilization of $y_1y_2\cdots y_m=z_1z_2\cdots z_l$ which also can be interpreted as a sequence of rational blowdowns. See Example~\ref{examcase1} below.


\begin{remark}
\label{remark1}
The minimal resolution graph of $X$ shows that there is a symplectic cobordism between a lens space determined by $X'$ and the link of $X$. Therefore, the minimal symplectic filling of $X$ correponding to $Y$ is obtained from the minimal symplectic filling of $X'$ corresponding to $Y'$ by symplectically gluing the cobordism. Since every minimal sympelctic filling of a lens space is obtained from a sequence of rational blowdowns from the minimal resolution, we can conclude that every minimal symplectic filling corresponding to a $P$-resolution $Y$ in \nameref{case1} is obtained via a sequence of rational blowdowns from the minimal resoluiton.
\end{remark}

\begin{example}
\label{examcase1}
Let $(X,0)$ be a tetrahedral singularity of type $T_{6(5-2)+5}$ which has the following $P$-resolution $Y$:
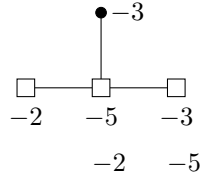
\begin{center}
\begin{tikzpicture}
\draw (0,0)--(2,0) (1,0)--(1,1);
\node [draw, fill=white, shape=rectangle, anchor=center] at (0,0) {};
\node [draw, fill=white, shape=rectangle, anchor=center] at (1,0) {};
\node [draw, fill=white, shape=rectangle, anchor=center] at (2,0) {};
 \filldraw (1,1) circle (2pt);
\draw (1,1) node[right] {$-3$};
\draw (0,-0.15) node[below] {$-2$};
\draw (1,-0.15) node[below] {$-5$};
\draw (2,-0.15) node[below] {$-3$};
\end{tikzpicture}
\end{center} 
Thus, a PALF on a $P$-resolution $Y' =$
\begin{tikzpicture}
\draw (0,0)--(2,0); 
\node [draw, fill=white, shape=rectangle, anchor=center] at (0,0) {};
\node [draw, fill=white, shape=rectangle, anchor=center] at (1,0) {};
\node [draw, fill=white, shape=rectangle, anchor=center] at (2,0) {};
\draw (0,0.15) node[above] {$-2$};
\draw (1,0.15) node[above] {$-5$};
\draw (2,0.15) node[above] {$-3$};
\end{tikzpicture}
of 
\begin{tikzpicture}
\draw (0,0)--(2,0); 
\filldraw (0,0) circle (2pt);
\filldraw (1,0) circle (2pt);
\filldraw (2,0) circle (2pt);

\draw (0,0.15) node[above] {$-2$};
\draw (1,0.15) node[above] {$-5$};
\draw (2,0.15) node[above] {$-3$};
\end{tikzpicture}
can be obtained via a monodromy substitution of the form
$$\alpha_1^2\alpha_2\alpha_3\alpha_4\alpha_5\gamma_4\gamma_5
=xyzw\gamma_3$$
which is shown in Figure~\ref{Tetrahedral} below.
\begin{figure}[htb]
\begin{center}
\begin{tikzpicture}[scale=1.0]
\begin{scope}
\begin{knot}[
	clip width=5,
	clip radius = 2pt,
	end tolerance = 1pt,
]
\strand (0,0.5)--(0,-3.5);
\strand (1,0.5)--(1,-3.5);
\strand (2,0.5)--(2,-3.5);
\strand (3,0.5)--(3,-3.5);
\strand (4,0.5)--(4,-3.5);




\strand[red, thick] (0.25,-0.5) to [out=right, in=right] (0.25,-0.75)--(-0.25,-0.75) to [out=left, in=left] (-0.25,-0.5)--cycle;

\strand[red, thick] (0.25,-1.25) to [out=right, in=right] (0.25,-1.5)--(-0.25,-1.5) to [out=left, in=left] (-0.25,-1.25)--cycle;

\strand[red, thick] (1.25,-0.95) to [out=right, in=right] (1.25,-1.2)--(0.75,-1.2) to [out=left, in=left] (0.75,-0.95)--cycle;
\strand[red, thick] (2.25,-0.95) to [out=right, in=right] (2.25,-1.2)--(1.75,-1.2) to [out=left, in=left] (1.75,-0.95)--cycle;
\strand[red, thick] (3.25,-0.95) to [out=right, in=right] (3.25,-1.2)--(2.75,-1.2) to [out=left, in=left] (2.75,-0.95)--cycle;
\strand[red, thick] (4.25,-0.95) to [out=right, in=right] (4.25,-1.2)--(3.75,-1.2) to [out=left, in=left] (3.75,-0.95)--cycle;
\strand[red, thick] (3.25,-2) to [out=right, in=right] (3.25,-2.25)--(-0.25,-2.25) to [out=left, in=left] (-0.25,-2)--cycle;

\strand[red, thick] (4.25,-2.75) to [out=right, in=right] (4.25,-3)--(-0.25,-3) to [out=left, in=left] (-0.25,-2.75)--cycle;

\flipcrossings{1,3,5,7,9,11,13,15,17,19,21,23,25,27,29}
\end{knot}
\filldraw (0,0) circle (1.5pt);
\filldraw (1,0) circle (1.5pt);
\filldraw (2,0) circle (1.5pt);
\filldraw (3,0) circle (1.5pt);
\filldraw (4,0) circle (1.5pt);
\draw (0.35,-0.5) node[right] {$\alpha_1$};

\draw (1.25,-0.75) node[right] {$\alpha_2$};
\draw (2.25,-0.75) node[right] {$\alpha_3$};
\draw (3.25,-0.75) node[right] {$\alpha_4$};
\draw (4.25,-0.75) node[right] {$\alpha_5$};
\draw (3.35,-2) node[right] {$\gamma_4$};
\draw (4.35,-2.75) node[right] {$\gamma_5$};

\draw[->] (4.5,-1.75)--(5.25,-1.75);
\end{scope}

\begin{scope}[shift={(6,0.25)}]
\begin{knot}[
	clip width=5,
	clip radius = 2pt,
	end tolerance = 1pt,
]
\strand (0,0.5)--(0,-4);
\strand (1,0.5)--(1,-4);
\strand (2,0.5)--(2,-4);
\strand (3,0.5)--(3,-4);
\strand (4,0.5)--(4,-4);

\strand[red, thick] (4.25,-0.5) to [out=right, in=right] (4.25,-0.75)--(0.75,-0.75) to [out=left, in=left] (0.75,-0.5)--cycle;

\strand[red, thick] (4.25,-1.25) to [out=right, in=right] (4.25,-1.5)--(-0.25,-1.5) to [out=left, in=left] (-0.25,-1.25)--cycle;

\strand[red, thick] (3.25,-2) to [out=right, in=right] (3.25,-2.25)--(-0.25,-2.25) to [out=left, in=left] (-0.25,-2)--cycle;
\strand[red, thick] (3.25,-2.75) to [out=right, in=right] (3.25,-3)--(-0.25,-3) to [out=left, in=left] (-0.25,-2.75)--cycle;

\strand[red, thick] (2.25,-3.5) to [out=right, in=right] (2.25,-3.75)--(-0.25,-3.75) to [out=left, in=left] (-0.25,-3.5)--cycle;

\flipcrossings{9,19,29,37,1,39,3,23,33,5,15,35,7,17,27}
\end{knot}
\filldraw (0,0) circle (1.5pt);
\filldraw (1,0) circle (1.5pt);
\filldraw (2,0) circle (1.5pt);
\filldraw (3,0) circle (1.5pt);
\filldraw (4,0) circle (1.5pt);
\draw (4.35,-0.5) node[right] {$x$};
\draw (4.35,-1.25) node[right] {$y$};
\draw (3.35,-2) node[right] {$z$};
\draw (3.35,-2.75) node[right] {$w$};
\draw (2.35,-3.5) node[right] {$\gamma_3$};
\end{scope}
\end{tikzpicture}
\end{center}
\caption{A PALF on a $P$-resolution $Y'$.}
\label{Tetrahedral}
\end{figure}
Hence, we obtain a PALF $\beta_1\beta_2\widetilde{x}
\widetilde{y}\widetilde{z}\widetilde{w}\widetilde{\gamma}_3$ on the $P$-resolution $Y$ from a PALF $\beta_1\beta_2\widetilde{\alpha}_1^2\widetilde{\alpha}_2
\widetilde{\alpha}_3\widetilde{\alpha}_4\widetilde{\alpha}_5
\widetilde{\gamma}_4\widetilde{\gamma}_5$ on the minimal resolution of $X$ via a monodromy substitution of the form
$$\beta_2\widetilde{\alpha}_1^2\widetilde{\alpha}_2
\widetilde{\alpha}_3\widetilde{\alpha}_4\widetilde{\alpha}_5
\widetilde{\gamma}_4\widetilde{\gamma}_5=\beta_2\widetilde{x}
\widetilde{y}\widetilde{z}\widetilde{w}\widetilde{\gamma}_3,$$
which can be topologically interpreted as a rational blowdown surgery. See Figure~\ref{Partial-Y}.
\begin{figure}[h]
\begin{center}
\begin{tikzpicture}[scale=1.0]

\begin{scope}
\begin{knot}[
	clip width=5,
	clip radius = 2pt,
	end tolerance = 1pt,
]
\strand (0,0.5)--(0,-3.5);
\strand (0.75,0.5)--(0.75,-3.5);
\strand (1.25,0.5)--(1.25,-3.5);

\strand (2,0.5)--(2,-3.5);
\strand (3,0.5)--(3,-3.5);
\strand (4,0.5)--(4,-3.5);




\strand[blue, thick,rounded corners](0.875,-0.5)--(0.875,-0.75)--(0.625,-0.75)--(0.625,-0.5)--cycle;
\strand[blue, thick,rounded corners](1.375,-0.5)--(1.375,-0.75)--(1.125,-0.75)--(1.125,-0.5)--cycle;

\strand[red, thick] (0.25,-0.5) to [out=right, in=right] (0.25,-0.75)--(-0.25,-0.75) to [out=left, in=left] (-0.25,-0.5)--cycle;

\strand[red, thick] (0.25,-1.25) to [out=right, in=right] (0.25,-1.5)--(-0.25,-1.5) to [out=left, in=left] (-0.25,-1.25)--cycle;

\strand[red, thick] (1.45,-1.25) to [out=right, in=right] (1.45,-1.5)--(0.55,-1.5) to [out=left, in=left] (0.55,-1.25)--cycle;
\strand[red, thick] (2.25,-0.95) to [out=right, in=right] (2.25,-1.2)--(1.75,-1.2) to [out=left, in=left] (1.75,-0.95)--cycle;
\strand[red, thick] (3.25,-0.95) to [out=right, in=right] (3.25,-1.2)--(2.75,-1.2) to [out=left, in=left] (2.75,-0.95)--cycle;
\strand[red, thick] (4.25,-0.95) to [out=right, in=right] (4.25,-1.2)--(3.75,-1.2) to [out=left, in=left] (3.75,-0.95)--cycle;
\strand[red, thick] (3.25,-2) to [out=right, in=right] (3.25,-2.25)--(-0.25,-2.25) to [out=left, in=left] (-0.25,-2)--cycle;

\strand[red, thick] (4.25,-2.75) to [out=right, in=right] (4.25,-3)--(-0.25,-3) to [out=left, in=left] (-0.25,-2.75)--cycle;

\flipcrossings{1,3,5,7,9,11,13,15,17,19,21,23,25,27,29,31,33,35,37,39}
\end{knot}
\filldraw (0,0) circle (1.5pt);
\filldraw (0.75,0) circle (1.5pt);
\filldraw (1.25,0) circle (1.5pt);
\filldraw (2,0) circle (1.5pt);
\filldraw (3,0) circle (1.5pt);
\filldraw (4,0) circle (1.5pt);
\draw (0.5,-0.25) node {$\beta_1$};
\draw (1.5,-0.25) node {$\beta_2$};

\draw (-0.25,-0.5) node[left] {$\widetilde{\alpha}_1$};

\draw (1.25,-1.75) node[right] {$\widetilde{\alpha}_2$};
\draw (2.25,-0.75) node[right] {$\widetilde{\alpha}_3$};
\draw (3.25,-0.75) node[right] {$\widetilde{\alpha}_4$};
\draw (4.25,-0.75) node[right] {$\widetilde{\alpha}_5$};
\draw (3.35,-2) node[right] {$\widetilde{\gamma}_4$};
\draw (4.35,-2.75) node[right] {$\widetilde{\gamma}_5$};

\draw[->] (4.5,-1.75)--(5.25,-1.75);
\end{scope}

\begin{scope}[shift={(6,0)}]
\begin{knot}[
	clip width=5,
	clip radius = 2pt,
	end tolerance = 1pt,
]
\strand (0,1)--(0,-4);
\strand (0.75,1)--(0.75,-4);
\strand (1.25,1)--(1.25,-4);

\strand (2,1)--(2,-4);
\strand (3,1)--(3,-4);
\strand (4,1)--(4,-4);

\strand[blue, thick,rounded corners](0.875,0)--(0.875,-0.25)--(0.625,-0.25)--(0.625,-0)--cycle;
\strand[blue, thick,rounded corners](1.375,-0)--(1.375,-0.25)--(1.125,-0.25)--(1.125,-0)--cycle;

\strand[red, thick] (4.25,-0.5) to [out=right, in=right] (4.25,-0.75)--(0.5,-0.75) to [out=left, in=left] (0.5,-0.5)--cycle;

\strand[red, thick] (4.25,-1.25) to [out=right, in=right] (4.25,-1.5)--(-0.25,-1.5) to [out=left, in=left] (-0.25,-1.25)--cycle;

\strand[red, thick] (3.25,-2) to [out=right, in=right] (3.25,-2.25)--(-0.25,-2.25) to [out=left, in=left] (-0.25,-2)--cycle;
\strand[red, thick] (3.25,-2.75) to [out=right, in=right] (3.25,-3)--(-0.25,-3) to [out=left, in=left] (-0.25,-2.75)--cycle;

\strand[red, thick] (2.25,-3.5) to [out=right, in=right] (2.25,-3.75)--(-0.25,-3.75) to [out=left, in=left] (-0.25,-3.5)--cycle;

\flipcrossings{9,21,11,23,33,43,51,1,53,3,37,47,5,17,29,49,7,19,31,41}
\end{knot}
\filldraw (0,0.5) circle (1.5pt);
\filldraw (0.75,0.5) circle (1.5pt);
\filldraw (1.25,0.5) circle (1.5pt);
\filldraw (2,0.5) circle (1.5pt);
\filldraw (3,0.5) circle (1.5pt);
\filldraw (4,0.5) circle (1.5pt);
\draw (0.5,0.25) node {$\beta_1$};
\draw (1.5,0.25) node {$\beta_2$};

\draw (4.35,-0.5) node[right] {$\widetilde{x}$};
\draw (4.35,-1.25) node[right] {$\widetilde{y}$};
\draw (3.35,-2) node[right] {$\widetilde{z}$};
\draw (3.35,-2.75) node[right] {$\widetilde{w}$};
\draw (2.35,-3.5) node[right] {$\widetilde{\gamma}_3$};
\end{scope}
\end{tikzpicture}
\end{center}
\caption{A PALF on $P$-resolution $Y$.}
\label{Partial-Y}
\end{figure}
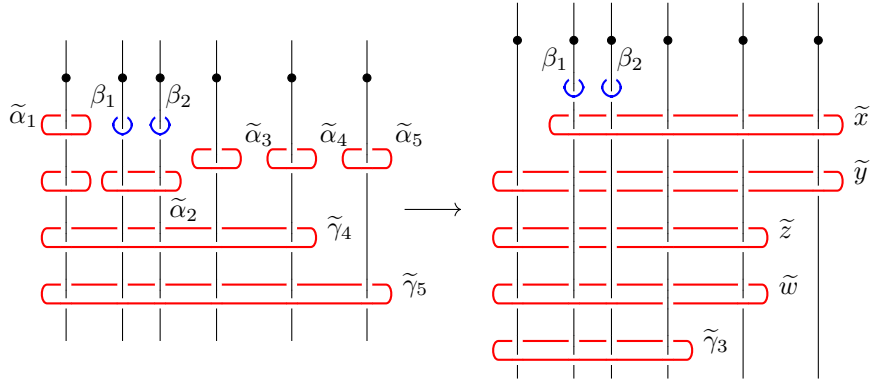

\end{example}

\subsection{Case 2}
In this subsection, we address a $P$-resolution $Y$ of $X$ such that any maximal linear subgraph $\Gamma_L$ of the minimal resolution of $Y$ cannot contain all singularities of class $T$ lying in $Y$. If the genus of generic fiber is $0$, then such a $P$-resolution $Y$ contains one of the subgraphs $\Gamma_i$ shown in Figure ~\ref{FNB}, which will be discussed in \ref{case2genus0}.
If the genus is $1$, then there are two such $P$-resolutions, which will be discussed in \ref{case2genus1}. See ~\cite{HJS} for the list of all $P$-resolutions for non-cyclic singularities.
\subsubsection{Genus-$0$ cases}
\label{case2genus0}
\begin{figure}[htb]
\begin{center}
\begin{tikzpicture}[scale=0.75]
\begin{scope}
\draw (0,0)--(2,0) (1,0)--(1,2);
\node [draw, fill=white, shape=rectangle, anchor=center] at (0,0) {};
\node [draw, fill=white, shape=rectangle, anchor=center] at (1,0) {};
\node [draw, fill=white, shape=rectangle, anchor=center] at (2,0) {};
\node [draw, fill=white, shape=rectangle, anchor=center] at (1,2) {};
 \draw (0,2) node {$\Gamma_1$};
 \filldraw (1,1) circle (2pt);

\draw (1,2) node[right] {$-4$};
\draw (1,1) node[right] {$-1$};
\draw (0,-0.15) node[below] {$-2$};
\draw (1,-0.15) node[below] {$-5$};
\draw (2,-0.15) node[below] {$-3$};
\end{scope}

\begin{scope}[shift={(4,0)}]
\draw (0,2) node {$\Gamma_2$};
\draw (0,0)--(2,0) (1,0)--(1,3);
\node [draw, fill=white, shape=rectangle, anchor=center] at (0,0) {};
\node [draw, fill=white, shape=rectangle, anchor=center] at (1,0) {};
\node [draw, fill=white, shape=rectangle, anchor=center] at (2,0) {};
 \node [draw, fill=white, shape=rectangle, anchor=center] at (1,2) {};
 \node [draw, fill=white, shape=rectangle, anchor=center] at (1,3) {};
 
 \filldraw (1,1) circle (2pt);

\draw (1,3) node[right] {$-5$};
\draw (1,2) node[right] {$-2$};
\draw (1,1) node[right] {$-1$};

\draw (0,-0.15) node[below] {$-2$};
\draw (1,-0.15) node[below] {$-5$};
\draw (2,-0.15) node[below] {$-3$};
\end{scope}

\begin{scope}[shift={(8,0)}]
\draw (0,2) node {$\Gamma_3$};
\draw (0,0)--(2,0) (1,0)--(1,3);
\node [draw, fill=white, shape=rectangle, anchor=center] at (0,0) {};
\node [draw, fill=white, shape=rectangle, anchor=center] at (1,0) {};
\node [draw, fill=white, shape=rectangle, anchor=center] at (2,0) {};
 \node [draw, fill=white, shape=rectangle, anchor=center] at (1,2) {};
 \node [draw, fill=white, shape=rectangle, anchor=center] at (1,3) {};
 
 \filldraw (1,1) circle (2pt);

\draw (1,3) node[right] {$-3$};
\draw (1,2) node[right] {$-3$};
\draw (1,1) node[right] {$-1$};

\draw (0,-0.15) node[below] {$-2$};
\draw (1,-0.15) node[below] {$-5$};
\draw (2,-0.15) node[below] {$-3$};

\end{scope}

\begin{scope}[shift={(0.5,-3.5)}]
\draw (0,2) node {$\Gamma_4$};
\draw (0,0)--(3,0) (1,0)--(1,2);
\node [draw, fill=white, shape=rectangle, anchor=center] at (0,0) {};
\node [draw, fill=white, shape=rectangle, anchor=center] at (1,0) {};
\node [draw, fill=white, shape=rectangle, anchor=center] at (3,0) {};
\node [draw, fill=white, shape=rectangle, anchor=center] at (1,2) {};

 \filldraw (2,0) circle (2pt);
 \filldraw (1,1) circle (2pt);

\draw (1,2) node[right] {$-4$};
\draw (1,1) node[right] {$-1$};
\draw (0,-0.15) node[below] {$-2$};
\draw (1,-0.15) node[below] {$-5$};
\draw (3,-0.15) node[below] {$-4$};
\draw (2,-0.15) node[below] {$-1$};

\end{scope}

\begin{scope}[shift={(5.5,-3.5)}]
\draw (0,2) node {$\Gamma_5$};
\draw (0,0)--(4,0) (1,0)--(1,2);
\node [draw, fill=white, shape=rectangle, anchor=center] at (0,0) {};
\node [draw, fill=white, shape=rectangle, anchor=center] at (1,0) {};
 \filldraw (2,0) circle (2pt);
 \filldraw (1,1) circle (2pt);
\node [draw, fill=white, shape=rectangle, anchor=center] at (3,0) {};
 \node [draw, fill=white, shape=rectangle, anchor=center] at (4,0) {};
 \node [draw, fill=white, shape=rectangle, anchor=center] at (1,2) {};
   
\draw (1,2) node[right] {$-4$};
\draw (1,1) node[right] {$-1$};
\draw (0,-0.15) node[below] {$-2$};
\draw (1,-0.15) node[below] {$-5$};
\draw (2,-0.15) node[below] {$-1$};
\draw (3,-0.15) node[below] {$-3$};
\draw (4,-0.15) node[below] {$-3$};

\end{scope}
\begin{scope}[shift={(5.5,-7)}]
\draw (0,0)--(4,0) (1,0)--(1,2);
\node [draw, fill=white, shape=rectangle, anchor=center] at (0,0) {};
\node [draw, fill=white, shape=rectangle, anchor=center] at (1,0) {};
\node [draw, fill=white, shape=rectangle, anchor=center] at (2,0) {};
 \node [draw, fill=white, shape=rectangle, anchor=center] at (1,2) {};
 \node [draw, fill=white, shape=rectangle, anchor=center] at (4,0) {};
 
 \filldraw (1,1) circle (2pt);
 \filldraw (3,0) circle (2pt);

\draw (1,2) node[right] {$-4$};
\draw (1,1) node[right] {$-1$};

\draw (0,-0.15) node[below] {$-2$};
\draw (1,-0.15) node[below] {$-5$};
\draw (2,-0.15) node[below] {$-3$};
\draw (3,-0.15) node[below] {$-1$};
\draw (4,-0.15) node[below] {$-4$};
\draw (0,2) node {$\Gamma_7$};
\end{scope}

\begin{scope}[shift={(0.5,-7)}]
\draw (0,2) node {$\Gamma_6$};
\draw (0,0)--(3,0) (1,0)--(1,2);
\node [draw, fill=white, shape=rectangle, anchor=center] at (0,0) {};
\node [draw, fill=white, shape=rectangle, anchor=center] at (1,0) {};
\node [draw, fill=white, shape=rectangle, anchor=center] at (3,0) {};
\node [draw, fill=white, shape=rectangle, anchor=center] at (2,0) {};
\node [draw, fill=white, shape=rectangle, anchor=center] at (1,2) {};

 \filldraw (1,1) circle (2pt);

\draw (1,2) node[right] {$-4$};
\draw (1,1) node[right] {$-1$};
\draw (0,-0.15) node[below] {$-2$};
\draw (1,-0.15) node[below] {$-6$};
\draw (3,-0.15) node[below] {$-3$};
\draw (2,-0.15) node[below] {$-2$};

\end{scope}

\end{tikzpicture}

\end{center}
\caption{Types of subgraph $\Gamma_i$.}
\label{FNB}
\end{figure}
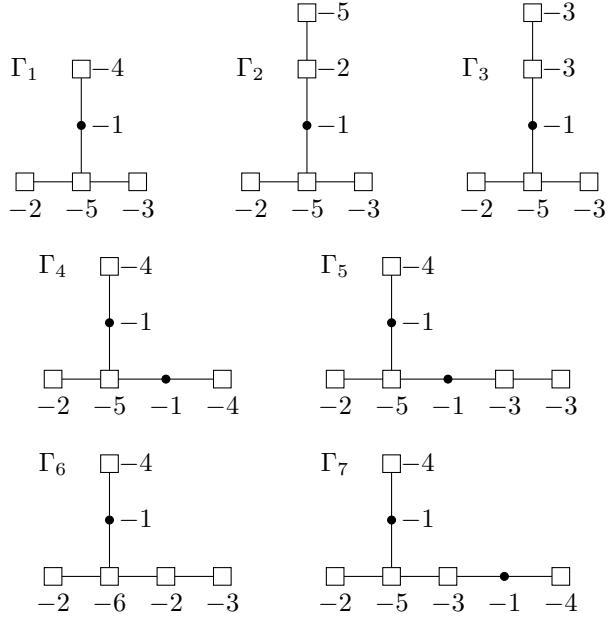

Note that each subgraph $\Gamma_i$ in Figure~\ref{FNB} represents a $P$-resolution $Y_i$ of another quotient surface singularity, say $X_i$. Since the subgraphs in Figure~\ref{FNB} contain all singularities of class $T$ in $Y$, it suffices to find an explicit PALF on $Y_i$. 
Recall that the minimal symplectic filling of $X_i$ corresponding to the $P$-resolution $Y_i$ can be obtained as follows. 
First, we rationally blow down all singularities of class $T$ lying in $Y_i$, except the one containing a central vertex. This yields a $4$-manifold diffeomorphic to the minimal resolution of $X_i$ that can be also obtained from $X_{\Gamma_i}$ by blowing down all $(-1)$-spheres until there is no $(-1)$-sphere in the resulting plumbing graph. Since the blow-ups and blow-downs can be performed in symplectic category, the above argument implies that there is a convex plumbing of symplectic submanifolds of codimension $0$ in the minimal resolution of $X_i$ according to the dual graph of singularities of class $T$ containing the central vertex of $\Gamma_i$. The desired minimal symplectic filling of $X_i$ is obtained by rationally blowing down the convex plumbing. 
Hence, in order to obtain a PALF on $Y_i$, we only need to find a subword representing the convex plumbing from the monodormy factorization of the minimal resolution of $X_i$. 
This is always possible because we know explicitly how vanishing cycles (i.e., $2$-handles) in the PALF on the minimal resolution of $X_i$ correspond to vertices (i.e., embedded 2-spheres) in the minimal resolution graph of $X_i$.

\begin{figure}[htb]
\begin{center}
\begin{tikzpicture}[scale=0.9]
\begin{scope}
\draw (0,0)--(2,0) (1,0)--(1,2);
\node [draw, fill=white, shape=rectangle, anchor=center] at (0,0) {};
\node [draw, fill=white, shape=rectangle, anchor=center] at (1,0) {};
\node [draw, fill=white, shape=rectangle, anchor=center] at (2,0) {};
\node [draw, fill=white, shape=rectangle, anchor=center] at (1,2) {};
 \draw (0,2) node {$\Gamma_1$};
 \filldraw (1,1) circle (2pt);

\draw (1,2) node[right] {$-4$};
\draw (1,1) node[right] {$-1$};
\draw (0,-0.15) node[below] {$-2$};
\draw (1,-0.15) node[below] {$-5$};
\draw (2,-0.15) node[below] {$-3$};
\draw[->] (3.5,0.5)--node[above] {$-4$ curve}(4.5,0.5);
\draw (4,1.2) node {rationally blowing down};
\end{scope}

\begin{scope}[shift={(6,0)}]
\draw (0,0)--(2,0) (1,0)--(1,1);

 \filldraw (1,1) circle (2pt);
 \filldraw (0,0) circle (2pt);
 \filldraw (1,0) circle (2pt);
 \filldraw (2,0) circle (2pt);
 
\draw (0,2) node {$X_1$};

\draw (1,1) node[right] {$-3$};

\draw (0,-0.15) node[below] {$-2$};
\draw (1,-0.15) node[below] {$-4$};
\draw (2,-0.15) node[below] {$-3$};
\end{scope}

\end{tikzpicture}
\begin{tikzpicture}[scale=0.7]
\begin{scope}
\begin{knot}[
	clip width=5,
	clip radius = 2pt,
	end tolerance = 1pt,
]

\draw (2,6) node {};
\strand (0,0)--(0,5);
\strand (1,0)--(1,5);
\strand (2,0)--(2,5);
\strand (3,0)--(3,5);
\strand (4,0)--(4,5);
\strand[red,thick] (-0.25,4.125)--(0.25,4.125) to [out=right, in=right] (0.25,3.975)--(-0.25,3.975) to [out=left, in=left] (-0.25,4.125);
\strand[red,thick] (-0.25,3.125)--(0.25,3.125) to [out=right, in=right] (0.25,2.975)--(-0.25,2.975) to [out=left, in=left] (-0.25,3.125);

\strand[red,thick] (0.75,3.625)--(1.25,3.625) to [out=right, in=right] (1.25,3.475)--(0.75,3.475) to [out=left, in=left] (0.75,3.625);
\strand[red,thick] (1.75,3.625)--(2.25,3.625) to [out=right, in=right] (2.25,3.475)--(1.75,3.475) to [out=left, in=left] (1.75,3.625);
\strand[red,thick] (2.75,3.625)--(3.25,3.625) to [out=right, in=right] (3.25,3.475)--(2.75,3.475) to [out=left, in=left] (2.75,3.625);
\strand[red,thick] (3.75,3.625)--(4.25,3.625) to [out=right, in=right] (4.25,3.475)--(3.75,3.475) to [out=left, in=left] (3.75,3.625);

\strand (0.75,2.625)--(2.25,2.625) to [out=right, in=right] (2.25,2.475)--(0.75,2.475) to [out=left, in=left] (0.75,2.625);

\strand[red,thick] (-0.25,1.625)--(3.25,1.625) to [out=right, in=right] (3.25,1.475)--(-0.25,1.475) to [out=left, in=left] (-0.25,1.625);

\strand[red,thick] (-0.25,0.625)--(4.25,0.625) to [out=right, in=right] (4.25,0.475)--(-0.25,0.475) to [out=left, in=left] (-0.25,0.625);

\flipcrossings{2,4,10,18,26,32,12,20,6,14,22,28,8,16,24,30,34}
\end{knot}
\filldraw (0,4.5) circle (1.5pt);
\filldraw (1,4.5) circle (1.5pt);
\filldraw (2,4.5) circle (1.5pt);
\filldraw (3,4.5) circle (1.5pt);
\filldraw (4,4.5) circle (1.5pt);

\draw (-0.25,4) node[left] {$\alpha_1$};
\draw (-0.25,1.5) node[left] {$\gamma_4$};

\draw (-0.25,0.5) node[left] {$\gamma_5$};
\draw (0.75,2.5) node[left] {$\beta$};
\draw (1.5,3.8) node {$\alpha_2$};
\draw (2.5,3.8) node {$\alpha_3$};
\draw (3.5,3.8) node {$\alpha_4$};
\draw (4.5,3.8) node {$\alpha_5$};
\draw[->] (4.5,2.5)--(5.5,2.5);
\end{scope}
\begin{scope}[shift={(6.25,0.5)}]
\begin{knot}[
	clip width=5,
	clip radius = 2pt,
	end tolerance = 1pt,
]

\draw (2,6) node {};
\strand (0,-1)--(0,5);
\strand (1,-1)--(1,5);
\strand (2,-1)--(2,5);
\strand (3,-1)--(3,5);
\strand (4,-1)--(4,5);
\strand (0.75,4.125)--(2.25,4.125) to [out=right, in=right] (2.25,3.975)--(0.75,3.975) to [out=left, in=left] (0.75,4.125);

\strand[red,thick] (0.75,3.625)--(4.25,3.625) to [out=right, in=right] (4.25,3.475)--(0.75,3.475) to [out=left, in=left] (0.75,3.625);
\strand[red,thick] (-0.25,2.625)--(4.25,2.625) to [out=right, in=right] (4.25,2.475)--(-0.25,2.475) to [out=left, in=left] (-0.25,2.625);
\strand[red,thick] (-0.25,1.625)--(3.25,1.625) to [out=right, in=right] (3.25,1.475)--(-0.25,1.475) to [out=left, in=left] (-0.25,1.625);
\strand[red,thick] (-0.25,0.625)--(3.25,0.625) to [out=right, in=right] (3.25,0.475)--(-0.25,0.475) to [out=left, in=left] (-0.25,0.625);
\strand[red,thick] (-0.25,-0.375)--(2.25,-0.375) to [out=right, in=right] (2.25,-0.525)--(-0.25,-0.525) to [out=left, in=left] (-0.25,-0.375);

\flipcrossings{10,22,12,24,34,42,2,44,4,28,38,6,18,40,8,20,32}
\end{knot}

\draw (2.25,4) node[right] {$\beta$};
\draw (4.25,3.5) node[right] {$x$};
\draw (4.25,2.5) node[right] {$y$};
\draw (3.25,1.5) node[right] {$z$};
\draw (3.25,0.5) node[right] {$w$};
\draw (2.25,-0.5) node[right] {$\gamma_3$};

\filldraw (0,4.5) circle (1.5pt);
\filldraw (1,4.5) circle (1.5pt);
\filldraw (2,4.5) circle (1.5pt);
\filldraw (3,4.5) circle (1.5pt);
\filldraw (4,4.5) circle (1.5pt);

\end{scope}
\end{tikzpicture}

\end{center}
\caption{A PALF on $\Gamma_1$.}
\label{find}
\end{figure}
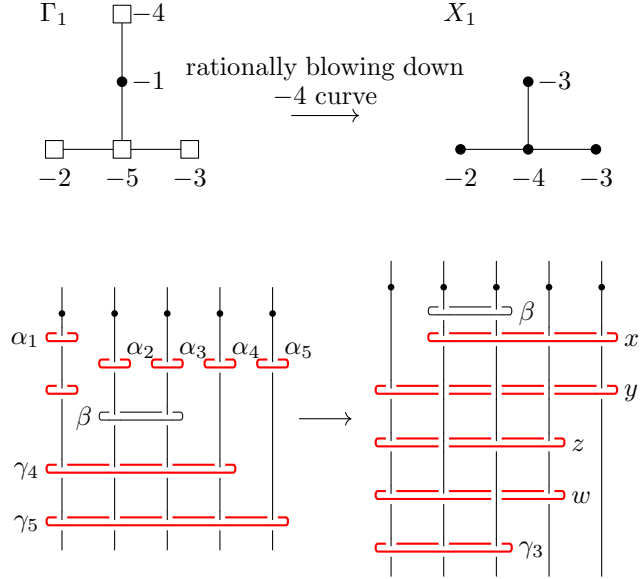

\begin{example}
Figure \ref{find} shows a PALF on the minimal resolution of $X_1$ whose monodromy factorization is given by 
$$\alpha_1^2\alpha_2\alpha_3\alpha_4\alpha_5\beta\gamma_4
  \gamma_5.$$  
Note that the monodromy factorization without $\beta$ above 
represents a $4$-manifold diffeomorphic to a convex plumbing of a subgraph
$$\begin{tikzpicture}
\draw(0,0)--(2,0);
\filldraw(0,0) circle (2pt);
\filldraw(1,0) circle (2pt);
\filldraw(2,0) circle (2pt);
\draw (0,0) node[above] {$-2$};
\draw (1,0) node[above] {$-5$};
\draw (2,0) node[above] {$-3$};
\end{tikzpicture}$$
lying in $\Gamma_1$. Hence, the right-hand side of Figure~\ref{find} above yields a desired PALF 
$\beta xyz w\gamma_3$ on $Y_1$.
\end{example}

\subsubsection{Genus-$1$ cases} 
\label{case2genus1}
There are two $P$-resolutions corresponding to genus $1$ in Case 2, which come from an icosahedral singularity of type $I_{30(2-2)+29}$ and an octahedral singularity of type $O_{12(2-2)+11}$. See Figure~\ref{Two-except}.

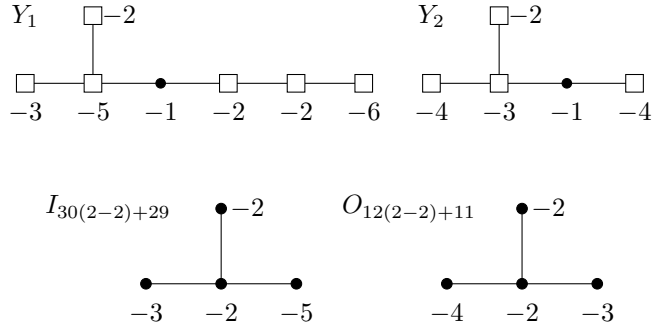
\begin{figure}[hbp]
\begin{tikzpicture}[scale=0.9]
\begin{scope}
\draw (0,1) node {$Y_1$};
\draw (0,0)--(5,0) (1,0)--(1,1);
\node [draw, fill=white, shape=rectangle, anchor=center] at (0,0) {};
\node [draw, fill=white, shape=rectangle, anchor=center] at (1,0) {};
\node [draw, fill=white, shape=rectangle, anchor=center] at (3,0) {};
\node [draw, fill=white, shape=rectangle, anchor=center] at (4,0) {};
\node [draw, fill=white, shape=rectangle, anchor=center] at (5,0) {};
\node [draw, fill=white, shape=rectangle, anchor=center] at (1,1) {};
\filldraw (2,0) circle (2pt);
\draw (1,1) node[right] {$-2$};
\draw (0,-0.15) node[below] {$-3$};
\draw (1,-0.15) node[below] {$-5$};
\draw (2,-0.15) node[below] {$-1$};
\draw (3,-0.15) node[below] {$-2$};
\draw (4,-0.15) node[below] {$-2$};
\draw (5,-0.15) node[below] {$-6$};
\end{scope}
\begin{scope}[shift={(6,0)}]
\draw (0,1) node {$Y_2$};
\draw (0,0)--(3,0) (1,0)--(1,1);
\node [draw, fill=white, shape=rectangle, anchor=center] at (0,0) {};
\node [draw, fill=white, shape=rectangle, anchor=center] at (1,0) {};
\node [draw, fill=white, shape=rectangle, anchor=center] at (3,0) {};
\node [draw, fill=white, shape=rectangle, anchor=center] at (1,1) {};
\filldraw (2,0) circle (2pt);
\draw (1,1) node[right] {$-2$};
\draw (0,-0.15) node[below] {$-4$};
\draw (1,-0.15) node[below] {$-3$};
\draw (2,-0.15) node[below] {$-1$};
\draw (3,-0.15) node[below] {$-4$};
\end{scope}
\end{tikzpicture}

\vspace{2 em}

\begin{tikzpicture}
\begin{scope}
\draw (0,0)--(2,0) (1,0)--(1,1);
\filldraw (0,0) circle (2pt);
\filldraw (1,0) circle (2pt);
\filldraw (2,0) circle (2pt);
\filldraw (1,1) circle (2pt);
\draw (-0.5,1) node {$I_{30(2-2)+29}$};
\draw (1,1) node[right] {$-2$};
\draw (0,-0.15) node[below] {$-3$};
\draw (1,-0.15) node[below] {$-2$};
\draw (2,-0.15) node[below] {$-5$};

\end{scope}
\begin{scope}[shift={(4,0)}]
\draw (0,0)--(2,0) (1,0)--(1,1);
\filldraw (0,0) circle (2pt);
\filldraw (1,0) circle (2pt);
\filldraw (2,0) circle (2pt);
\filldraw (1,1) circle (2pt);
\draw (-0.5,1) node {$O_{12(2-2)+11}$};
\draw (1,1) node[right] {$-2$};
\draw (0,-0.15) node[below] {$-4$};
\draw (1,-0.15) node[below] {$-2$};
\draw (2,-0.15) node[below] {$-3$};
\end{scope}
\end{tikzpicture}
\caption{Two genus-$1$ cases.}
\label{Two-except}
\end{figure}

As shown in \ref{case2genus0}, we first rationally blow down all singularities of class $T$, except the one containing a central vertex such that the resulting $4$-manifold is diffeomorphic to the minimal resolution.
Since there is a bad vertex in the minimal resolution graph, we must consider a genus-$1$ PALF on this minimal resolution graph.

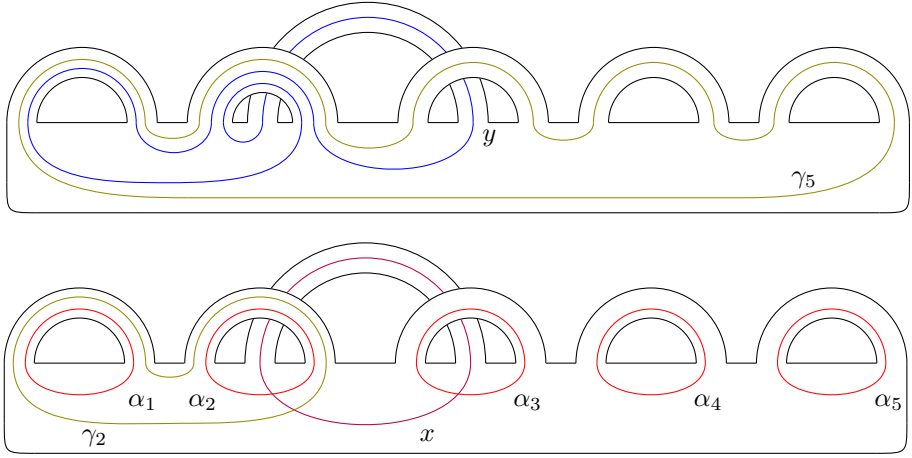
\begin{figure}[htbp]
\begin{tikzpicture}[scale=0.4]

\draw (0,0)--(-1,0);
\draw (-1,0) arc (0:180:2.5 and 2.5);
\draw (-2,0) arc (0:180:1.5 and 1.5);
\draw (-2,0)--(-5,0);

\draw(12,0)--(13,0);
\draw (13+5,0) arc (0:180:2.5 and 2.5);
\draw (13+4,0) arc (0:180:1.5 and 1.5);
\draw (14,0)--(14+3,0);

\draw(18,0)--(19,0);
\draw (19+5,0) arc (0:180:2.5 and 2.5);
\draw (19+4,0) arc (0:180:1.5 and 1.5);
\draw (20,0)--(20+3,0);

\draw (10,0) arc (0:180:4 and 4);
\draw (9,0) arc (0:180:3 and 3);
\draw[blue] (9.5,0) arc (0:180:3.5 and 3.5);
\filldraw[fill=white,white] (0,0)--(1.5,0) (1.5,0) arc (180:0:1 and 1) (3.5,0)--(5,0) (5,0) arc (0:180:2.5 and 2.5);
\filldraw[fill=white,white] (7,0)--(1+7,0) (1+7,0) arc (180:0:1.5 and 1.5) (4+7,0)--(5+7,0) (5+7,0) arc (0:180:2.5 and 2.5);
\draw (1.5,0)--(2,0) (3,0)--(3.5,0) (5,0)--(7,0) (8,0)--(9,0) (10,0)--(11,0);
\draw (5,0) arc (0:180:2.5 and 2.5);
\draw (3.5,0) arc (0:180:1 and 1);
\draw (5+7,0) arc (0:180:2.5 and 2.5);
\draw (4+7,0) arc (0:180:1.5 and 1.5);

\draw[blue] (4.2,0) arc (0:180:1.7 and 1.7);
\draw[blue] (3.8,0) arc (0:180:1.3 and 1.3);
\draw[blue] (2.5,0) arc (0:-180:0.65 and 0.65);
\draw[blue] (9.5,0) to [out=down, in=down] (4.2,0);
\draw[blue] (-1.7,0) arc (0:180:1.8);
\draw[blue] (-5.3,0) to [out=down, in=left] (-0.6,-2) to [out=right, in=down] (3.8,0);
\draw[blue] (-1.7,0) to [out=down, in=left] (-0.45,-1) to [out=right, in=down]  (0.8,0);
\draw  (9.5,-0.5) node[right] {$y$};
\draw[olive] (-1.4,0) arc (0:180: 2.1);
\draw[olive] (-1.4,0) to [out=down, in=down] (0.4,0);
\draw[olive] (4.6,0) arc (0:180: 2.1);
\draw[olive] (4.6,0) to [out=down, in=down] (7.5,0);
\draw[olive] (11.5,0) arc (0:180: 2);
\draw[olive] (11.5,0) to [out=down, in=down] (13.5,0);
\draw[olive] (17.5,0) arc (0:180: 2);
\draw[olive] (17.5,0) to [out=down, in=down] (19.5,0);
\draw[olive] (23.5,0) arc (0:180: 2);
\draw[olive] (-5.6,0) to [out=down, in=left] (0,-2.5);
\draw[olive] (23.5,0) to [out=down, in=right] (18.5,-2.5);
\draw[olive] (0,-2.5)--(18.5,-2.5);
\draw (20.5,-2.5) node[above] {$\gamma_5$};

\draw  (-6,0)--(-6,-2) (-5,-3) -- (23,-3) (24,-2) -- (24,0);
\draw (-6,-2)..controls +(0,-1) and +(-1,0) ..(-5,-3);
\draw (23,-3)..controls +(1,0) and +(0,-1) ..(24,-2);
\
\end{tikzpicture}

\vspace{1 em}

\begin{tikzpicture}[scale=0.4]

\draw (0,0)--(-1,0);
\draw (-1,0) arc (0:180:2.5 and 2.5);
\draw (-2,0) arc (0:180:1.5 and 1.5);
\draw (-2,0)--(-5,0);

\draw(12,0)--(13,0);
\draw (13+5,0) arc (0:180:2.5 and 2.5);

\draw (13+4,0) arc (0:180:1.5 and 1.5);
\draw (14,0)--(14+3,0);

\draw(18,0)--(19,0);
\draw (19+5,0) arc (0:180:2.5 and 2.5);
\draw (19+4,0) arc (0:180:1.5 and 1.5);
\draw (20,0)--(20+3,0);

\draw (10,0) arc (0:180:4 and 4);
\draw (9,0) arc (0:180:3 and 3);
\draw[purple] (9.5,0) arc (0:180:3.5 and 3.5);
\draw[purple] (9.5,0) to [out=down, in=down] (2.5,0);
\filldraw[fill=white,white] (0,0)--(1,0) (1,0) arc (180:0:1.5 and 1.5) (4,0)--(5,0) (5,0) arc (0:180:2.5 and 2.5);
\filldraw[fill=white,white] (7,0)--(1+7,0) (1+7,0) arc (180:0:1.5 and 1.5) (4+7,0)--(5+7,0) (5+7,0) arc (0:180:2.5 and 2.5);
\draw (1,0)--(2,0) (3,0)--(4,0) (5,0)--(7,0) (8,0)--(9,0) (10,0)--(11,0);
\draw (5,0) arc (0:180:2.5 and 2.5);
\draw[olive] (4.7,0) arc (0:180:2.2);
\draw[olive] (-1.3,0) arc (0:180:2.2);
\draw (4,0) arc (0:180:1.5 and 1.5);
\draw (5+7,0) arc (0:180:2.5 and 2.5);
\draw (4+7,0) arc (0:180:1.5 and 1.5);

\draw[olive] (4.7,0) to [out=down,in=right] (0,-2) to [out=left,in=down] (-5.7,0);
\draw[olive] (0.3,0) to [out=down,in=down] (-1.3,0);

\draw  (7.5,-2.4) node[right] {$x$};
\draw  (-3,-2.5) node {$\gamma_2$};
\draw[red] (4.3,0) arc (0:180: 1.8);
\draw[red] (4.3,0) to [out=down,in=down] (0.7,0);
\draw  (0.6,-1.3) node {$\alpha_2$};
\draw[red] (23.3,0) arc (0:180: 1.8);
\draw[red] (23.3,0) to [out=down,in=down] (19.7,0);
\draw  (23.4,-1.3) node {$\alpha_5$};
\draw[red] (-1.7,0) arc (0:180: 1.8);
\draw[red] (-1.7,0) to [out=down,in=down] (-5.3,0);
\draw  (-1.4,-1.3) node {$\alpha_1$};

\draw[red] (17.3,0) arc (0:180: 1.8);
\draw[red] (17.3,0) to [out=down,in=down] (13.7,0);
\draw  (17.4,-1.3) node {$\alpha_4$};

\draw[red] (11.3,0) arc (0:180: 1.8);
\draw[red] (11.3,0) to [out=down,in=down] (7.7,0);
\draw  (11.4,-1.3) node {$\alpha_3$};



\draw  (-6,0)--(-6,-2) (-5,-3) -- (23,-3) (24,-2) -- (24,0);
\draw (-6,-2)..controls +(0,-1) and +(-1,0) ..(-5,-3);
\draw (23,-3)..controls +(1,0) and +(0,-1) ..(24,-2);

\end{tikzpicture}

\caption{A genus $1$-PALF on the minimal resolution of $I_{30(2-2)+29}$.}
\label{PALF-T}
\end{figure}

\emph{$I_{30(2-2)+29}$ case}: Using the same technique yields a monodromy factorization of the following form for the minimal resolution of $I_{30(2-2)+29}$
$$xy\gamma_2^2\alpha_1\alpha_2\alpha_3\alpha_4\alpha_5\gamma_5,$$
where $\alpha_i$ and $\gamma_i$ are curves encircling the $i^{th}$ hole and th first $i$ holes, respectively (refer to~Figure~\ref{PALF-T}). 
Note that, using Hurwitz moves and $y=t_{\alpha_2}(t_{\gamma_2}(x))$, we can change the monodromy factorization as follows:
\begin{eqnarray*}
&\phantom{0}& \hspace{-2 em} xy\gamma_2^2\alpha_1\alpha_2\alpha_3\alpha_4\alpha_5\gamma_5\\ 
&\sim& x\alpha_2\gamma_2 x\alpha_3\alpha_1\alpha_4\alpha_5\gamma_2\gamma_5\\
&\sim&t_x(\alpha_2)t_x(\gamma_2)x^2\alpha_3\alpha_1\alpha_4\alpha_5\gamma_2\gamma_5\\&\sim&t_x(\alpha_2)t_x(\gamma_2)t_x^2(\alpha_3)x^2\alpha_1\alpha_4\alpha_5\gamma_2\gamma_5\\
&\sim&t_x^2(\alpha_3)(t_x^2\cdot t_{\alpha_3}^{-1}\cdot t_x^{-1})(\alpha_2)(t_x^2\cdot t_{\alpha_3}^{-1}\cdot t_x^{-1})(\gamma_2)x^2\alpha_1\alpha_4\alpha_5\gamma_2\gamma_5.
\end{eqnarray*}
Now, taking a global conjugation of each monodromy with $f=t_x\cdot t_{\alpha_3}\cdot t_x^{-2}$ and using a braid relation 
$t_x\cdot t_{\alpha_3}\cdot t_x =t_{\alpha_3}\cdot t_x \cdot t_{\alpha_3}$, we can show that the global monodromy factorization becomes
$$t_x(\alpha_3)\alpha_2\gamma_2\alpha_3^2\alpha_1\alpha_4
\alpha_5f(\gamma_2)\gamma_5.$$

Since the subword $\alpha_2\gamma_2\alpha_3^2\alpha_1\alpha_4\alpha_5\gamma_5$
in the monodromy factorization above corresponds to\begin{tikzpicture}
\draw (0,0)--(2,0);
\filldraw (0,0) circle (2pt);
\filldraw (1,0) circle (2pt);
\filldraw (2,0) circle (2pt);
\draw (0,0) node[above] {$-3$};
\draw (1,0) node[above] {$-5$};
\draw (2,0) node[above] {$-2$};
\end{tikzpicture}
lying in the minimal resolution graph, we obtain a PALF on $Y_1$ by rationally blowing down it. \vspace{0.5 ex}

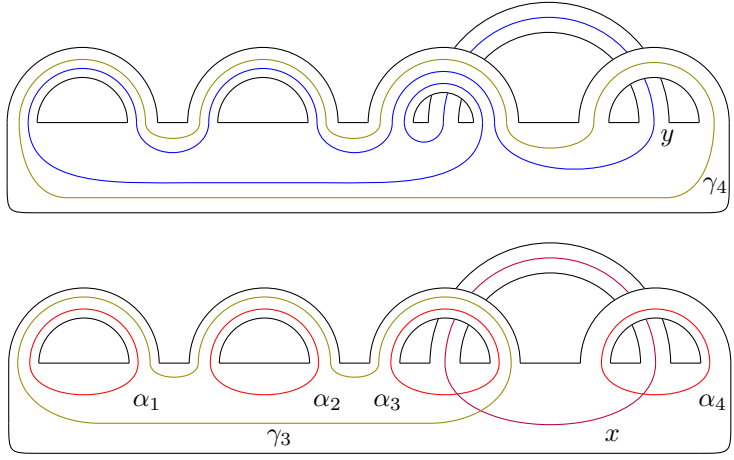
\begin{figure}[htbp]
\begin{tikzpicture}[scale=0.4]

\draw (0,0)--(-1,0);
\draw (-1,0) arc (0:180:2.5 and 2.5);
\draw (-2,0) arc (0:180:1.5 and 1.5);
\draw (-2,0)--(-5,0);

\draw (-6,0)--(-7,0);
\draw (-7,0) arc (0:180:2.5 and 2.5);
\draw (-8,0) arc (0:180:1.5 and 1.5);
\draw (-8,0)--(-11,0);



\draw (10,0) arc (0:180:4 and 4);
\draw (9,0) arc (0:180:3 and 3);
\draw[blue] (9.5,0) arc (0:180:3.5 and 3.5);
\filldraw[fill=white,white] (0,0)--(1.5,0) (1.5,0) arc (180:0:1 and 1) (3.5,0)--(5,0) (5,0) arc (0:180:2.5 and 2.5);
\filldraw[fill=white,white] (7,0)--(1+7,0) (1+7,0) arc (180:0:1.5 and 1.5) (4+7,0)--(5+7,0) (5+7,0) arc (0:180:2.5 and 2.5);
\draw (1.5,0)--(2,0) (3,0)--(3.5,0) (5,0)--(7,0) (8,0)--(9,0) (10,0)--(11,0);
\draw (5,0) arc (0:180:2.5 and 2.5);
\draw (3.5,0) arc (0:180:1 and 1);
\draw (5+7,0) arc (0:180:2.5 and 2.5);
\draw (4+7,0) arc (0:180:1.5 and 1.5);

\draw[blue] (4.2,0) arc (0:180:1.7 and 1.7);
\draw[blue] (3.8,0) arc (0:180:1.3 and 1.3);
\draw[blue] (2.5,0) arc (0:-180:0.65 and 0.65);
\draw[blue] (9.5,0) to [out=down, in=down] (4.2,0);
\draw[blue] (-1.7,0) arc (0:180:1.8);
\draw[blue] (-7.7,0) arc (0:180:1.8);
\draw[blue] (-7.7,0) to [out=down, in=left] (-6.5,-1) to [out=right, in=down]  (-5.3,0);

\draw[blue] (-11.3,0) to [out=down, in=left] (-6,-2)--(-3.75,-2)--(-1.5,-2) to [out=right, in=down](3.8,0);
\draw[blue] (-1.7,0) to [out=down, in=left] (-0.45,-1) to [out=right, in=down]  (0.8,0);
\draw  (9.4,-0.5) node[right] {$y$};
\draw[olive] (-1.4,0) arc (0:180: 2.1);
\draw[olive] (-1.4,0) to [out=down, in=down] (0.4,0);
\draw[olive] (-7.4,0) arc (0:180: 2.1);
\draw[olive] (-7.4,0) to [out=down, in=down] (-5.6,0);
\draw[olive] (-11.6,0) to [out=down, in=left] (-10,-2.5)--(9.9,-2.5) to [out=right, in=down] (11.5,0);
\draw[olive] (4.6,0) arc (0:180: 2.1);
\draw[olive] (4.6,0) to [out=down, in=down] (7.5,0);
\draw[olive] (11.5,0) arc (0:180: 2);
\draw (11.55,-2.7) node [above] {$\gamma_4$};

\draw  (-12,0)--(-12,-2) (-11,-3) -- (11,-3) (12,-2)--(12,0);
\draw (-12,-2)..controls +(0,-1) and +(-1,0) ..(-11,-3);
\draw (11,-3)..controls +(1,0) and +(0,-1) ..(12,-2);

\end{tikzpicture}

\vspace{1 em}

\begin{tikzpicture}[scale=0.4]

\draw (0,0)--(-1,0);
\draw (-1,0) arc (0:180:2.5 and 2.5);
\draw (-2,0) arc (0:180:1.5 and 1.5);
\draw (-2,0)--(-5,0);

\draw (-6,0)--(-7,0);
\draw (-7,0) arc (0:180:2.5 and 2.5);
\draw (-8,0) arc (0:180:1.5 and 1.5);
\draw (-8,0)--(-11,0);



\draw (10,0) arc (0:180:4 and 4);
\draw (9,0) arc (0:180:3 and 3);
\draw[purple] (9.5,0) arc (0:180:3.5 and 3.5);
\draw[purple] (9.5,0) to [out=down, in=down] (2.5,0);
\filldraw[fill=white,white] (0,0)--(1,0) (1,0) arc (180:0:1.5 and 1.5) (4,0)--(5,0) (5,0) arc (0:180:2.5 and 2.5);
\filldraw[fill=white,white] (7,0)--(1+7,0) (1+7,0) arc (180:0:1.5 and 1.5) (4+7,0)--(5+7,0) (5+7,0) arc (0:180:2.5 and 2.5);
\draw (1,0)--(2,0) (3,0)--(4,0) (5,0)--(7,0) (8,0)--(9,0) (10,0)--(11,0);
\draw (5,0) arc (0:180:2.5 and 2.5);
\draw[olive] (4.7,0) arc (0:180:2.2);
\draw[olive] (-1.3,0) arc (0:180:2.2);
\draw (4,0) arc (0:180:1.5 and 1.5);
\draw (5+7,0) arc (0:180:2.5 and 2.5);
\draw (4+7,0) arc (0:180:1.5 and 1.5);

\draw[olive] (4.7,0) to [out=down,in=right] (2,-2)--(-9,-2) to [out=left,in=down] (-11.7,0);
\draw[olive] (0.3,0) to [out=down,in=down] (-1.3,0);
\draw[olive] (-7.3,0) arc (0:180: 2.2);
\draw[olive] (-7.3,0) to [out=down, in=down] (-5.7,0);

\draw  (7.5,-2.4) node[right] {$x$};
\draw  (-3,-2.5) node {$\gamma_3$};
\draw[red] (4.3,0) arc (0:180: 1.8);
\draw[red] (4.3,0) to [out=down,in=down] (0.7,0);
\draw  (0.6,-1.3) node {$\alpha_3$};

\draw[red] (-7.7,0) arc (0:180: 1.8);
\draw[red] (-7.7,0) to [out=down,in=down] (-11.3,0);
\draw  (-7.4,-1.3) node {$\alpha_1$};

\draw[red] (-1.7,0) arc (0:180: 1.8);
\draw[red] (-1.7,0) to [out=down,in=down] (-5.3,0);
\draw  (-1.4,-1.3) node {$\alpha_2$};


\draw[red] (11.3,0) arc (0:180: 1.8);
\draw[red] (11.3,0) to [out=down,in=down] (7.7,0);
\draw  (11.4,-1.3) node {$\alpha_4$};



\draw  (-12,0)--(-12,-2) (-11,-3) -- (11,-3) (12,-2)--(12,0);
\draw (-12,-2)..controls +(0,-1) and +(-1,0) ..(-11,-3);
\draw (11,-3)..controls +(1,0) and +(0,-1) ..(12,-2);

\end{tikzpicture}

\caption{A genus $1$-PALF on the minimal resolution of $O_{12(2-2)+11}$.}
\label{PALF-O}
\end{figure}

\emph{$O_{12(2-2)+11}$ case}: Starting from a PALF on \begin{tikzpicture}
\draw (0,0)--(2,0);
\filldraw (0,0) circle (2pt);
\filldraw (1,0) circle (2pt);
\filldraw (2,0) circle (2pt);
\draw (0,0) node[above] {$-4$};
\draw (1,0) node[above] {$-2$};
\draw (2,0) node[above] {$-3$};
\end{tikzpicture}, we obtain a monodromy factorization for the minimal resolution of
$O_{12(2-2)+11}$ as follows (see Figure~\ref{PALF-O} for vanishing cycles):
$$xy\gamma_3^2\alpha_1\alpha_2\alpha_3\alpha_4\gamma_4$$ 
A similar computation shows that above monodromy factorization is equivalent to
$$t_x(\alpha_4)\alpha_3\gamma_3\alpha_4^2\alpha_1\alpha_2 f(\gamma_3)\gamma_4,$$
where $f=t_x\cdot t_{\alpha_4} \cdot t_x^{-2}$. Now we can construct a PALF on $Y_2$ because the subword $\alpha_3 \gamma_3\alpha_4^2\alpha_1\alpha_2\gamma_4$ corresponds to \begin{tikzpicture}
\draw (0,0)--(2,0);
\filldraw (0,0) circle (2pt);
\filldraw (1,0) circle (2pt);
\filldraw (2,0) circle (2pt);
\draw (0,0) node[above] {$-4$};
\draw (1,0) node[above] {$-3$};
\draw (2,0) node[above] {$-2$};
\end{tikzpicture}.

Hence, summarizing all the arguments in this section, we conclude:

\begin{thm}
There is an explicit algorithm for a genus-$0$ or genus-$1$ PALF on any minimal symplectic filling of the link of non-cyclic quotient surface singularities.
\end{thm}

Recall that we divided all $P$-resolutions of $X$ into two families in the construction of PALF on each $P$-resolution $Y$: those with and without a maximal subgraph $\Gamma_L$ containing all singularities of class $T$ in $Y$. The algorithm of PALF for the first family is essentially the same algorithm for cyclic cases, which means that the Milnor fiber corresponding to a $P$-resolution $Y$ is obtained topologically via rational blowdowns from the minimal resolution of $X$~\cite{BOz}. 
On the other hand, we found a subword diffeomorphic to a convex
neighborhood of a linear chain of $2$-spheres in a smooth 4-manifold whose boundary is $L(p^2, pq-1)$ for the second family, which also can be rationally blowdown~\cite{EnMV}. Hence we have:

\begin{cor}
Any Milnor fiber of the link of quotient surface singularities can be obtained, up to diffeomorphism, via a sequence of rational blowdowns from the minimal resolution of the singularity.
\end{cor}

\medskip


\providecommand{\bysame}{\leavevmode\hbox to3em{\hrulefill}\thinspace}

\end{document}